\documentclass[reqno,english]{amsart}
\usepackage[utf8]{inputenc}
\usepackage[normalem]{ulem}
\usepackage{amsmath}
\usepackage{amsthm}
\usepackage{amsfonts}
\usepackage{amssymb}
\usepackage{indentfirst}
\usepackage{graphicx}
\usepackage{float}
\numberwithin{equation}{section}
\newtheorem{theorem}{Theorem}[section]
\newtheorem{lemma}[theorem]{Lemma}
\newtheorem{corollary}[theorem]{Corollary}
\newtheorem{proposition}[theorem]{Proposition}

\newtheorem{remark}{Remark}[section]

\usepackage{mathtools}
\usepackage{thmtools}
\declaretheoremstyle[headfont=\normalfont]{normalhead}
\usepackage{color}
\usepackage{morefloats}
\usepackage{hyperref}
\usepackage{textcomp}
\usepackage{url}
\title[Sharp Beckner's Inequality]{On Beckner's Inequality for Axially Symmetric Functions on $\mathbb{S}^6$}

\author{Changfeng Gui}
\address{Department of Mathematics, University of Macau, Taipan, Macau  and Department of Mathematics, University of Texas at San Antonio, Texas, USA}
\email{changfeng.gui@utsa.edu}

\author{Tuoxin Li}
\address{Department of Mathematics, University of British Columbia,Vancouver, Canada}
\email{tuoxin@math.ubc.ca}

\author{Juncheng Wei}
\address{Department of Mathematics, University of British Columbia,Vancouver, Canada}
\email{jcwei@math.ubc.ca}

\author{Zikai Ye}
\address{Department of Mathematics, University of British Columbia,Vancouver, Canada}
\email{yezikai@math.ubc.ca}
\begin{document}

\maketitle

\begin{abstract}
We prove that axially symmetric solutions to the $Q$-curvature type problem
$$ \alpha P_6 u + 120(1-\frac{e^{6u}}{\int_{\mathbb{S}^6} e^{6u}})=0 \ \ \ \ \ \mbox{on} \ \mathbb{S}^6 $$
must be constants, provided that $ \frac{1}{2}\le \alpha <1$. In view of the existence of non-constant solutions obtained by Gui-Hu-Xie \cite{GHW2022} for $\frac{1}{7}<\alpha<\frac{1}{2}$, this result is sharp. This result closes the gap of the related results in \cite{GHW2022}, which proved a similar uniqueness result for $\alpha \geq 0.6168$. The improvement is based on two  types of  new estimates: one is  a better estimate of the semi-norm $\lfloor G\rfloor^2$,  the other one is a family of refined estimates on Gegenbauer coefficients, such as pointwise decaying and cancellations properties.
\end{abstract}

\section{Introduction and Main Results}

Beckner's inequality on $\mathbb{S}^6$,  a higher order Moser-Trudinger inequality,  asserts that  the functional
\begin{equation*}
J_{\alpha}(u):=\frac{\alpha}{2}\int_{\mathbb{S}^6}u(P_{6}u)
\text{d}w+120\int_{\mathbb{S}^6}u
\text{d}w-20\ln\int_{\mathbb{S}^6}e^{6u}
\text{d}w
\end{equation*}
 is non-negative for $\alpha=1$ and all $u\in H^2(\mathbb{S}^6)$, where d$w$ denotes the normalized Lebesgue measure on $\mathbb{S}^6$ with $\int_{\mathbb{S}^6}\text{d}w=1$ and $P_6=-\Delta(-\Delta+4)(-\Delta+6)$ represents the Paneitz operator on $\mathbb{S}^6$. Additionally, with the extra assumption that the mass center of $u$ is at the origin and $u$ belongs to the set
\begin{equation*}
\mathcal{L}=\left\{u\in H^2(\mathbb{S}^6)\ :\ \int_{\mathbb{S}^6}e^{6u} x_j \text{d}w=0,\ j=1,..., 7 \right \},
\end{equation*}
an improved higher-order Moser-Trudinger-Onofri inequality demonstrates that for any $\alpha\geq\frac{1}{2}$, a constant $C(\alpha)\geq0$ exists such that $J_{\alpha}(u)\geq -C(\alpha)$. As in the second-order case \cite{ChangYang1987}, it is conjectured that $C(\alpha)$ can be chosen to be $0$ for any $\alpha\geq \frac{1}{2}$.

The functional $J_\alpha$'s Euler-Lagrange equation is the following $Q$-curvature-type equation on $\mathbb{S}^6$
\begin{equation}\label{paneitz}
\alpha P_6 u+120(1-\frac{e^{6u}}{\int_{\mathbb{S}^6}e^{6u}\text{d}w})=0 \ \mbox{on} \ \mathbb{S}^6,
\end{equation}
If \eqref{paneitz} admits only constant solutions, then the conjecture is valid. If $\alpha<1$ is near $1$, the third author and Xu \cite{WX1998} proved that all solutions to \eqref{paneitz} are constants. However, for general $\alpha\in[\frac{1}{2},1)$, it remains unresolved. For results and backgrounds on $Q$-curvature problems, we refer to \cite{ChangYang1995, ChangYang1997, DHL2000,DM2008,GHX2021, GurMal2015, LiXiong2019, Mal2006, WX1998} and the references therein.

The corresponding problem on $\mathbb{S}^2$ is known as the Nirenberg problem:
$$ -\alpha \Delta u + 1-  \frac{e^{2u}}{\int_{\mathbb{S}^2} e^{2u}}=0 \ \ \mbox{on} \ \mathbb{S}^2.$$

This problem has been extensively studied over the past four decades. For more information, refer to \cite{ChangYang1987, ChangYang1988, JinLiXiong2017} and the references therein. A. Chang and P. Yang conjectured in \cite{ChangYang1987, ChangYang1988} that the following functional
\begin{equation*}
\alpha\int_{\mathbb{S}^2}|\nabla u |^2\text{d}w+2\int_{\mathbb{S}^2} u\text{d}w- \ln \int_{\mathbb{S}^2}e^{2u} \text{d}w
\end{equation*}
is non-negative for any $\alpha\geq\frac{1}{2}$ and $u$ with zero center of mass $\int_{\mathbb{S}^2}e^{2u}\Vec{x}\text{d}w=0$.
Feldman, Froese, Ghoussoub and the first author  \cite{FFGG1998} demonstrated that the conjecture is true for axially symmetric functions when $\alpha >\frac{16}{25}-\epsilon$, the first  and the third  author  in  \cite{GW2000} confirmed  that the conjecture is indeed  true for axially symmetric functions.  Later Ghoussoub and Lin \cite{GL2010} showed that the conjecture holds true for $\alpha >\frac{2}{3}-\epsilon$. Finally,   the first author  and Moradifam \cite{GM2018} proved  the full  conjecture.
For more general results on improved Moser-Trudinger-Onofri inequality on $\mathbb{S}^2$ and its connections with the Szeg"o limit theorem, see \cite{ChangGui202, ChangHang2022}.

For the related problem on $\mathbb{S}^4$,
\begin{equation}
\alpha P_4 u+6(1-\frac{e^{4u}}{\int_{\mathbb{S}^4}e^{4u}\text{d}w})=0 \ \mbox{on} \ \mathbb{S}^4,
\end{equation}
various results have been achieved for axially symmetric solutions. Gui-Hu-Xie \cite{GHX2021} proved the existence of non-constant solutions for $\frac{1}{5}<\alpha<\frac{1}{2}$ using bifurcation methods. They also demonstrated that for $\alpha\geq 0.517$, the above equation admits only constant solutions with axially symmetric assumption. The precise bound $\alpha\geq \frac{1}{2}$ is obtained by Li-Wei-Ye \cite{LWY2022} using refined estimates on Gegenbauer polynomials.

These settings can be extended to the $\mathbb{S}^n$ case for any $n\geq 3$. Gui-Hu-Xie \cite{GHW2022} established the existence of non-constant solutions using bifurcation methods for $\frac{1}{n+1}<\alpha<\frac{1}{2}$,  while for $\alpha \geq 0.6168$ ($n=6$) and $ \alpha \geq 0.8261$ ($n=8$),  all critical points are constants.

In this paper, we focus on axially symmetric solutions in the $\mathbb{S}^6$ case for $\alpha\in [\frac{1}{2},1)$. As we will see later, the problem is considerably difficult.

As in \cite{GHW2022}, \eqref{paneitz} becomes:

\begin{equation}\label{axial}
-\alpha[(1-x^2)^3u']^{(5)}+120-128\frac{e^{6u}}{\gamma}=0,\ x\in(-1,1),
\end{equation}
which is the critical point of the functional
\begin{align}
I_\alpha(u)
&=-\frac{\alpha}{2}\int_{-1}^{1}(1-x^2)^2[(1-x^2)^3 u']^{(5)}u+120\int_{-1}^{1}(1-x^2)^2 u\nonumber\\
&-\frac{64}{3}\ln\left(\frac{15}{16}\int_{-1}^{1}(1-x^2)^2 e^{6u}\right)
\end{align}
restricted to the set
\begin{equation}
\mathcal{L}_r=\{u\in H^2(\mathbb{S}^6):\ u=u(x)\text{ and }\int_{-1}^{1}x(1-x^2)^2e^{6u} dx=0\}.
\end{equation}

The main result of this paper is:

\begin{theorem}\label{main}
If $\alpha\geq \frac{1}{2}$, then the only critical points of the functional $I_\alpha$ restricted to $\mathcal{L}_r$ are constant functions. As a consequence, we have the following improved Beckner's inequality for axially symmetric functions on $\mathbb{S}^6$
$$ \inf_{ u\in {\mathcal{L}_r}} I_\alpha (u)=0, \ \alpha \geq \frac{1}{2}. $$
\end{theorem}

In the work of Gui-Hu-Xie \cite{GHW2022}, the assumption $\alpha\geq\frac{1}{2}$ is shown to be sharp, and they proved Theorem \ref{main} for $\alpha\geq 0.6168$ using a strategy similar to that in \cite{GHX2021, GW2000, LWY2022}. Specifically, they expand $G=(1-x^2)u'$ in terms of Gegenbauer polynomials and introduce a quantity $D$ related to the Gegenbauer coefficients and the estimate of $\lfloor G\rfloor^2$ (see \eqref{Gfloor}). However, unlike the $\mathbb S^4$ case discussed in \cite{GHX2021}, they are unable to obtain a bound on $\beta$ and, consequently, on $a=\frac{6}{7}(1-\alpha\beta)$. As a result, they cannot use $D$ to generate a series of inequalities as in \cite{GHX2021} and proceed through the induction procedure.

In this paper, we provide a better estimate on $\lfloor G\rfloor^2$ and work with a revised quantity $D$. To render the induction procedure $a\leq \frac{d_0}{\lambda_n}$ feasible, we employ refined point-wise estimates of Gegenbauer polynomials similar to those in $\mathbb{S}^4$ \cite{LWY2022} to improve the estimates of $G$'s Gegenbauer coefficients. More precisely, we refine the decaying behavior of Gegenbauer polynomials near $x=\pm 1$. Additionally, we utilize the cancellation properties of consecutive Gegenbauer polynomials to modify  the  methods in the $\mathbb{S}^4$ case.

This paper is organized as follows. In Section 2, we gather some properties of Gegenbauer polynomials, expand $G$ in terms of Gegenbauer polynomials, and cite some basic facts from \cite{GHW2022}. In Section 3, we present improved estimates of $\lfloor G\rfloor^2$ and Gegenbauer coefficients of $G$. In Section 4, we prove Theorem \ref{main} using the estimates above. Several Lemmas in Section 3 and Proposition \ref{lambda=1} are proven in the appendices.

\section{Preliminaries and some basic estimates}
In this section, we first introduce some properties of Gegenbauer polynomials and some known facts about the equation.

The Gegenbauer polynomials of  order $\nu$ and degree $k$ (\cite{Mori1998}) is given by
\begin{equation*}
C_{k}^{\nu}(x)=\frac{(-1)^k}{2^k k!}\frac{\Gamma(\nu+\frac{1}{2})\Gamma(k+2\nu)}{\Gamma(2\nu)\Gamma(\nu+k+\frac{1}{2})}(1-x^2)^{-\nu+\frac{1}{2}}\frac{d^k}{dx^k} (1-x^2)^{k+\nu-\frac{1}{2}}.
\end{equation*}

$C_{k}^{\nu}$ is an even function if $k$ is even and it is odd if $k$ is odd. The derivative of $C_{k}^{\nu}$ satisfies
\begin{equation}\label{201}
\frac{d}{dx}C_{k}^{\nu}(x)=2\nu C_{k-1}^{\nu+1}(x).
\end{equation}

Let $F_k^\nu$ be the normalization of $C_{k}^{\nu}$ such that $F_k^\nu(1)=1$, i.e.
\begin{equation}\label{Fknu}
    F_k^\nu=\frac{k!\Gamma(2\nu)}{\Gamma(k+2\nu)} C_{k}^{\nu},
\end{equation}
then $F_k^\nu$ satisfies
\begin{equation}\label{DE}
(1-x^2)(F_{k}^\nu)''-(2\nu+1)x(F_k^\nu)'+k(k+2\nu)F_k^\nu=0,
\end{equation}
and \eqref{201} becomes
\begin{align}\label{derivative}
    (F_{k}^\nu)'=\frac{k(k+2\nu)}{2\nu+1}F_{k-1}^{\nu+1}.
\end{align}
It is also useful to introduce the following expressions using hypergeometric functions
    \begin{align}\label{hyper odd}
        F_{2m+1}^\nu(\cos \theta)=\cos\theta {_2}F_1(-m,m+\nu+1;\nu+\frac{1}{2};\sin^2\theta),
    \end{align}
     \begin{align}\label{hyper even}
        F_{2m}^{\nu+1}(\cos \theta)= {_2}F_1(-m,m+\nu+1;\nu+\frac{3}{2};\sin^2\theta),
    \end{align}
where we recall the hypergeometric function is defined for $|x|<1$ by power series
    \begin{equation*}
        {_2}F_1(a,b;c;x)=\sum_{k=0}^\infty \frac{(a)_k(b)_k}{(c)_k}\frac{x^k}{k!}.
    \end{equation*}
Here $(a)_k=\frac{\Gamma(a+k)}{\Gamma(a)}$ is the Pochhammer symbol.

On $\mathbb{S}^6$, the corresponding Gegenbauer polynomial is $C_{k}^{\frac{5}{2}}$. For notational simplicity, in what follows we will write $F_k$ for $F_k^\frac{5}{2}$, and there should be no danger of confusion.

From \eqref{DE} it turns out that $F_k$ satisfies
\begin{equation}
(1-x^2)F_{k}''-6xF_k'+\lambda_k F_k=0
\end{equation}
and
\begin{equation}
\int_{-1}^{1}(1-x^2)F_{k}F_l=\frac{128}{(2k+5)(\lambda_k+4)(\lambda_k+6)}\delta_{kl},
\end{equation}
where $\lambda_k=k(k+5)$. As in \cite{GHX2021, GW2000}, we define the following key quantity
\begin{equation}
G(x)=(1-x^2)u',
\end{equation}
where $u$ is a solution to \eqref{axial}. Then $G$ satisfies the equation

\begin{equation}\label{G}
\alpha[(1-x^2)^2 G]^{(5)}+120-128\frac{e^{6u}}{\gamma}=0,
\end{equation}
where
\begin{equation}
    \gamma=\int_{-1}^{1}(1-x^2)^2 e^{6u}.
\end{equation}

$G$ can be expanded in terms of Gegenbauer polynomials
\begin{equation}
\label{Gexpand}
G=a_0F_0+\beta x+a_2F_2(x)+\sum_{k=3}^{\infty}a_kF_k(x).
\end{equation}

Denote
\begin{equation}
\label{gdef}
g= (1-x^2)^2 \frac{e^{6u}}{\gamma}, \ a:=\int_{-1}^1 (1-x^2)g.
\end{equation}

We recall some results from \cite{GHW2022}.
\begin{lemma}
For $g=(1-x^2)^2\frac{e^{6u}}{\gamma}$ and $G=(1-x^2)u'$ as above, we have $a_0=0$ and
\begin{equation}\label{}
\int_{-1}^{1}(1-x^2)F_1G=\frac{16}{105}\beta,
\end{equation}
\begin{equation}
a=\int_{-1}^{1}(1-x^2)g=\frac{6}{7}(1-\alpha\beta),
\end{equation}
\begin{equation}\label{bbk}
\int_{-1}^{1}(1-x^2)F_k G=-\frac{128}{\alpha(\lambda_k+4)(\lambda_k+6)}\int_{-1}^{1}(1-x^2)gF_{k}',\text{ }k\geq 2,
\end{equation}
\begin{equation}\label{G''}
\int_{-1}^{1}|[(1-x^2)^2G]''|^2=\frac{256}{35}(7-\frac{1}{\alpha})\beta.
\end{equation}
\end{lemma}

\begin{lemma}\label{G_j}
For all $x\in (-1,1)$, we have
\begin{equation}
G_j:=(-1)^j[(1-x^2)^j G]^{(2j+1)}\leq \frac{(2j+1)!}{\alpha},\ 0\leq j\leq 2.
\end{equation}
\end{lemma}

\section{Refined Estimates}
In this section, we deduce two refined estimates on the semi-norm $\lfloor G\rfloor^2$ and $b_k$ defined later.

To get a rough estimate of $\beta$ and $a=\frac{6}{7}(1-\alpha\beta)$, we need an estimate of the following semi-norm $\lfloor G\rfloor^2$. Let
\begin{equation}
\lfloor G\rfloor^2=-\int_{-1}^{1}(1-x^2)^2[(1-x^2)^3G']^{(5)}G.
\end{equation}

By integrating by parts (see Gui-Hu-Xie \cite{GHW2022}), we have
\begin{align}\label{Gfloor}
\lfloor G\rfloor^2
=&-15\int_{-1}^{1}|[(1-x^2)^2 G]''|^2+\frac{720}{\alpha}\int_{-1}^{1}(1-x^2)^2G^2+30\int_{-1}^{1}(1-x^2)^4G'(G'')^2\nonumber\\
&+160\int_{-1}^{1}(1-x^2)^3(G')^3.
\end{align}

With the help of Lemma \ref{G_j}, they applied $G'\leq \frac{1}{\alpha}$ directly to the last two integrals and obtained an estimate of $\lfloor G\rfloor^2$
\begin{equation*}
\lfloor G\rfloor^2
\leq(\frac{30}{\alpha}-15)\int_{-1}^{1}|[(1-x^2)^2 G]''|^2-\frac{320}{\alpha}\int_{-1}^{1}(1-x^2)^3(G')^2.
\end{equation*}

However, with this estimate, it is not enough to get a rough lower bound of $\beta$, hence an upper bound of $a$. The main issue here is that the coefficient of $\int_{-1}^{1}|[(1-x^2)^2 G]''|^2$ is too large. To solve this problem, we introduce the following Proposition to drop the third integral in \eqref{Gfloor}.

\begin{proposition}
\begin{equation}\label{floor}
\lfloor G\rfloor^2
\leq-15\int_{-1}^{1}|[(1-x^2)^2 G]''|^2+\frac{720}{\alpha}\int_{-1}^{1}(1-x^2)^2G^2+\frac{160}{\alpha}\int_{-1}^{1}(1-x^2)^3(G')^2,
\end{equation}
\end{proposition}
\begin{proof}
Integrating \eqref{Gfloor} by parts, we get
\begin{align*}
\lfloor G\rfloor^2
=&-15\int_{-1}^{1}|[(1-x^2)^2 G]''|^2+\frac{720}{\alpha}\int_{-1}^{1}(1-x^2)^2G^2+\int_{-1}^{1}(1-x^2)^3\Tilde{G}(G')^2,
\end{align*}
where
\begin{equation}
\Tilde{G}=-15(1-x^2)G'''+120xG''+160G'.
\end{equation}
Let
\begin{equation}
\hat{G}=-15(1-x^2)G'''+120xG''+150G'.
\end{equation}
Direct calculation yields that $\hat{G}$ satisfies
\begin{equation*}
(1-x^2)\hat{G}''-8x\hat{G}'-12\hat{G}=-15[(1-x^2)^2G]^{(5)}\geq -\frac{1800}{\alpha}.
\end{equation*}
The last inequality follows from Lemma \ref{G_j}.

Then we claim that
\begin{equation*}
\hat{G}\leq \frac{150}{\alpha}.
\end{equation*}

To prove the claim, denote $M=\max\limits_{-1\leq x\leq 1}\hat{G}(x)$.

If $M$ is attained at some point $x_0\in (-1,1)$, then
\begin{equation*}
\hat{G}'(x_0)=0,\ \hat{G}''(x_0)\leq 0
\end{equation*}
and the desired esitmate follows.

If $M$ is attained at $1$ or $-1$, without loss of generality, suppose there exists a sequence $x_k\to 1$ such that
\begin{equation*}
M=\lim\limits_{k\to\infty}\hat{G}(x_k).
\end{equation*}

Let $r=\sqrt{1-x^2}$ and write
\begin{equation*}
G(x)=\Bar{G}(r)\text{ and }u(x)=\Bar{u}(r)\text{ for }r\in [0,1),\ x\in (0,1].
\end{equation*}

Then we can extend $\Bar{u}(r)$ to be a smooth even function on $(-\frac{1}{2},\frac{1}{2})$.

Hence,
\begin{equation*}
G(x)=\Bar{G}(r)=-r\sqrt{1-r^2}u_r
\end{equation*}
is a smooth function.

Direct calculation yields that
\begin{equation*}
\Hat{G}(r)=-15(1-r^2)^2u_{rrrr}+30(1-r^2)(7r^2-4)\frac{u_{rrr}}{r}-15(48r^4-50r^2+5)\frac{u_{rr}}{r^2}
\end{equation*}
is an even function with respect to $r$. Moreover, since
\begin{equation*}
\lim\limits_{r\to 0}\frac{u_{rrr}(r)}{r}=u_{rrrr}(0),\ \lim\limits_{r\to 0}\frac{u_{rr}(r)}{r^2}=\frac{1}{2}u_{rrrr}(0),
\end{equation*}
$\hat{G}(r)$ is smooth on $(-\frac{1}{2},\frac{1}{2})$. Now we can write
\begin{align*}
\hat{G}(r)
&=c_1+c_2 r^2+c_3 r^4+O(r^6),\\
x\hat{G}'(x)&=-2c_2+O(r^2),\\
(1-x^2)\hat{G}''(x)&=(-2c_2+8c_3)r^2+O(r^4)
\end{align*}
near $r=0$. Since $\hat{G}(r)$ attains its local maximum at $r=0$, we have $c_2\leq 0$ and hence
\begin{equation*}
\lim\limits_{x\to 1}x\hat{G}'(x)\leq 0,\ \lim\limits_{x\to 1}(1-x^2)\hat{G}''(x)=0.
\end{equation*}

Then we obtain $M\leq \frac{150}{\alpha}$. Applying Lemma \ref{G_j} again, we get
\begin{equation*}
\Tilde{G}\leq \frac{160}{\alpha}.
\end{equation*}
and the Proposition follows.
\end{proof}

In the following part, we begin to estimate $b_{k}:=a_{k} \sqrt{ \int_{-1}^{1}(1-x^2)F_{k}^{2}}$, where $a_k$ is the $k$-th coefficient in the expansion of $G$ (see (\ref{Gexpand})). The estimates of $b_k$ play a key role in the proofs of \cite{GHX2021, GW2000}. In \cite{GHX2021}, they used \eqref{bbk} and the fact that
\begin{equation}
|F_k'(x)|\leq |F_k'(1)|=\frac{\lambda_k}{6}
\end{equation}
to estimate $b_k$ as follows
\begin{align*}
b_{k}^{2}&=a_{k}^2\int_{-1}^{1}(1-x^2)F_{k}^{2}=\frac{1}{\int_{-1}^{1}(1-x^2)F_{k}^{2}}\left[\frac{128}{\alpha\lambda_k}\int_{-1}^{1}(1-x^2)gF_{k}'\right]^2\\
&\leq \frac{(2k+5)(\lambda_k+4)(\lambda_k+6)}{128}\left[\frac{128}{\alpha\lambda_k(\lambda_k+2)}\frac{\lambda_k}{6}a\right]^2\\
&=\frac{32(2k+5)}{9\alpha^2(\lambda_k+4)(\lambda_k+6)}a^2.
\end{align*}

However, as in the $\mathbb{S}^4$ case, this estimate is not strong enough to deduce the induction
\begin{equation}
    a=\frac{6}{7}(1-\alpha\beta)\leq \frac{d_0}{\lambda_n}.
\end{equation}

Likewise, we need a refined estimate on $b_k$, which follows from the following refined estimate on Gegenbauer polynomials. For simplicity, in the rest of the paper, we denote
\begin{equation}
\Tilde{F}_{k}'=\frac{6}{\lambda_k}F_k'=\frac{720}{\lambda_k(\lambda_k+4)(\lambda_k+6)}C_{k-1}^{\frac{7}{2}}
\end{equation}
so that $\Tilde{F}_{k}'(1)=1$. As in $\mathbb{S}^4$, we split the integral in the right hand side of $b_k$ into two parts. To this end, we define
\begin{equation}\label{A}
a_+:=\int_{0}^{1}(1-x^2)g, \ a_-:=\int_{-1}^{0}(1-x^2)g,\ A_{k}^{+}=\int_{0}^{1}(1-x^2)\Tilde{F}_{k}'g,\ A_{k}^{-}=\int_{-1}^{0}(1-x^2)\Tilde{F}_{k}'g,
\end{equation}
 Without loss of generality, we may assume $a_+=\lambda a$ with $\frac{1}{2}\le \lambda \le 1$.

 Now we derive some estimates about $g$. Recalling the definition of $g$, we have
\begin{equation*}
\int_{-1}^{1}g=1,\ \int_{-1}^{1}xg=0\text{ and }\int_{-1}^{1}(1-x^2)g=a.
\end{equation*}

From the second integral above, we have
\begin{equation}\label{307}
\int_{0}^{1}g-\int_{0}^{1}(1-x)g=\int_{0}^{1}xg=-\int_{-1}^{0}xg=\int_{-1}^{0}g-\int_{-1}^{0}(1+x)g.
\end{equation}

Since
\begin{equation*}
\left|\int_{0}^{1}(1-x)g\right|\leq\int_{0}^{1}(1-x^2)g=a_+,\ \left|\int_{-1}^{0}(1+x)g\right|\leq\int_{0}^{1}(1-x^2)g=a_-,
\end{equation*}
combining with \eqref{307}, we have
\begin{equation*}
    \left|\int_{0}^{1}g-\int_{-1}^{0}g \right|\le a.
\end{equation*}
Hence
\begin{equation}\label{308}
\frac{1-a}{2}\leq\int_{0}^{1}g,\int_{-1}^{0}g\leq \frac{1+a}{2}.
\end{equation}
Moreover, it follows directly from the definition of $g$ that
\begin{equation}\label{309}
    \int_{0}^{1}xg\le \min\{\int_{0}^{1}g,\int_{-1}^{0}g\}\le\frac{1}{2},
\end{equation}
and
\begin{equation}\label{N310}
    \int_0^1 (1+x)g=1-\int_{-1}^0 (1+x)g<1.
\end{equation}

With the estimates on $g$ above, the following Theorem gives a refined estimate on $A_{k}^{\pm}$, hence on $b_k$.
\begin{theorem}\label{bk}
Let $d=8$, $b=0.33$. Suppose $a\leq \frac{16}{\lambda_n}$ for some $n\geq 3$. Then for all $k$, we have
\begin{equation}
|A_{k}^{+}|\leq\mathcal{A}_{k}^{+}:=
\begin{cases}
a_+-\frac{1-b}{d}\lambda_k a_{+}^{2},\text{ if }\lambda_k\leq \frac{\lambda_n}{4},\\
ba_{+}+(1-b)\frac{d}{4\lambda_{k}},\text{ if }\frac{\lambda_n}{4}<\lambda_k\leq \lambda_n,
\end{cases}
\end{equation}
\begin{equation}\label{A-}
|A_{k}^{-}|\leq\mathcal{A}_{k}^{-}:=
\begin{cases}
a_--\frac{1-b}{d}\lambda_k a_{-}^{2},\text{ if }a_{-}\leq\frac{4}{\lambda_n},\\
ba_{-}+(1-b)\frac{d}{4\lambda_{k}}\chi_{\{\lambda\neq 1\}},\text{ if }\frac{4}{\lambda_n}< a_-\leq\frac{8}{\lambda_n}.
\end{cases}
\end{equation}
\end{theorem}

In fact, for the toy cases in which $k$'s are small, better estimates can be obtained. The proof is left to Appendix \ref{appendix A}.

\begin{lemma}\label{A2-5}
For $A_k$, $2\leq k\leq 5$,
\begin{eqnarray}
    |A_2|&\le& a_+-a_+^2,\label{303}\\
    |A_3|&\le& a-\frac{9}{4}\frac{a^2}{a+1}(2\lambda^2-2\lambda+1),\label{304}\\
    |A_4|&\le&  (a_+-a_+^2)-\frac{11}{4} (a_+-a_+^2)^2+\frac{1}{4\sqrt{11}}a_-,\label{305}\\
    |A_5|&\le& a-\frac{11(a_+^2+a_-^2)}{2(a+1)}+\frac{143(a_+^3+a_-^3)}{10(a+1)^2}.\label{306}
\end{eqnarray}
\end{lemma}
Before we prove Theorem \ref{bk} for general $k$'s, we first introduce some point-wise estimates of Gegenbauer polynomials.

\begin{lemma}[Corollary 5.3 of Nemes and Olde Daalhuis \cite{Nemes2020} ]\label{lem33}
Let $0<\zeta<\pi$ and $N\geq 3$ be an  integer. Then
\begin{equation}\label{310}
C_{k-1}^{\frac{7}{2}}(\cos{\zeta})=\frac{2}{\Gamma(\frac{7}{2})(2\sin{\zeta})^\frac{7}{2}}\left(\sum_{n=0}^{N-1}t_n(3)\frac{\Gamma(k+6)}{\Gamma(k+n+\frac{7}{2})}\frac{\cos{(\delta_{k-1,n})}}{\sin^n{\zeta}}+R_N(\zeta,k-1)\right),
\end{equation}
where $\delta_{k,n}=(k+n+\frac{7}{2})\zeta-(\frac{7}{2}-n)\frac{\pi}{2}$, $t_n(\mu)=\frac{(\frac{1}{2}-\mu)_n(\frac{1}{2}+\mu)_n}{(-2)^n n!}$, and $(x)_n=\frac{\Gamma(x+n)}{\Gamma(x)}$ is the Pochhammer symbol. The remainder term $R$ satisfies the estimate
\begin{equation}\label{311}
|R_N(\zeta,k)|\leq |t_N(3)|\frac{\Gamma(k+6)}{\Gamma(k+N+\frac{7}{2})}\frac{1}{\sin^N{\zeta}}\cdot
\begin{cases}
|\sec{\zeta}|&\text{ if }0<\zeta\leq\frac{\pi}{4}\text{ or }\frac{3\pi}{4}\leq\zeta<\pi,\\
2\sin{\zeta}&\text{ if }\frac{\pi}{4}<\zeta<\frac{3\pi}{4}.
\end{cases}
\end{equation}
\end{lemma}

Using the pointwise estimate (\ref{310}),
we can prove the following lower and upper bounds for $\Tilde{F}_k'$. Recall that $\Tilde{F}_k'$ is odd for $k$ even and even for $k$ odd. It suffices to estimate $\Tilde{F}_k'$ on $[0,1]$. The proofs are left to  Appendix \ref{appendix B}.

\begin{lemma}\label{minFk}
Let $m_0=0.04$, then for all $k\ge 8$, we have
\begin{eqnarray*}
\widetilde{F}_k'\ge -m_0,\quad 0\le x\le 1.
\end{eqnarray*}
\end{lemma}

\begin{lemma}\label{ptwise}
Let $d=8$ and $b=0.33$. Then for all $k\ge 6$,
\begin{eqnarray*}
\widetilde{F}_k^{'}\le\left\{\begin{aligned} &b,\quad &0\le x\le1-\frac{d}{\lambda_k},\\
&1-\frac{\lambda_k}{d}(1-b)(1-x), \quad &1-\frac{d}{\lambda_k}\le x \le 1.\end{aligned}\right.
\end{eqnarray*}
\end{lemma}

With the help of the above two lemmas, we are able to derive Theoreo \ref{bk}.
\begin{proof}[Proof of Theorem \ref{bk}]
By \eqref{beta} below, we have $\beta\geq \frac{113}{88}$, $\alpha<0.578$ and hence $a\leq 0.221$. It is straightforward to check the cases when $2\leq k\leq 5$ hold for better estimate in the form of Lemma \ref{ptwise}. In the following argument, we may assume $k\geq 6$.
Define $I=(0,1-\frac{d}{\lambda_k})$, $II=(1-\frac{d}{\lambda_k},1)$, and $a_I=\int_I (1-x^2) g$, $a_{II}=\int_{II} (1-x^2) g$. Then by Lemma \ref{ptwise} and \eqref{N310}, we have
\begin{align}
    \int_0^1 (1-x^2)\widetilde{F}_k'g&=\int_I (1-x^2)\widetilde{F}_k'g+\int_{II} (1-x^2)\widetilde{F}_k'g\nonumber\\
    &\le \int_I (1-x^2)b g+\int_{II} (1-x^2)(1-\frac{\lambda_k}{d}(1-b)(1-x))g\nonumber\\
    &=ba_I+a_{II}-\frac{\lambda_k}{d}(1-b)\int_{II} (1-x^2)(1-x)g\nonumber\\
    &\le ba_I+a_{II}-\frac{\lambda_k}{d}(1-b)\frac{(\int_{II} (1-x^2)g)^2}{\int_{II} (1+x)g}\nonumber\\
    &\le ba_I+a_{II}-\frac{\lambda_k}{d}(1-b)a_{II}^2\nonumber\\
    &=ba_++(1-b)(a_{II}-\frac{\lambda_k}{d}a_{II}^2)\label{parabola}.
\end{align}

If $\lambda_k\leq \frac{\lambda_n}{4}$, we have $a_{II}\leq a_+\leq a\leq \frac{16}{\lambda_n}\leq \frac{d}{2\lambda_k}$. Hence,
\begin{equation*}
    \int_0^1 (1-x^2)\widetilde{F}_k'g\le a_++(1-b)(a_+-\frac{\lambda_k}{d}a_+^2)=a_+-\frac{\lambda_k}{d}(1-b)a_+^2.
\end{equation*}

For the case when $\lambda_k>\frac{\lambda_n}{4}$, we get directly
\begin{equation*}
\int_0^1 (1-x^2)\widetilde{F}_k'g\leq ba_++(1-b)\frac{d}{4\lambda_k}.
\end{equation*}

On the other hand, Lemma \ref{minFk} yields
\begin{equation*}
    \int_0^1 (1-x^2)\widetilde{F}_k'g\ge -0.04\int_0^1 (1-x^2)g=-0.04a_+.
\end{equation*}

Combining the above three estimates, we obtain the desired estimate on $A_{k}^{+}$.

Similarly, on estimating $A_{k}^{-}$, just note that $a_{-}\leq \frac{a}{2}\leq \frac{8}{\lambda_n}$. We can get an estimate analogous to \eqref{parabola} and then \eqref{A-} follows directly. We omit the details.
\end{proof}

Next we derive a uniform estimate of cancellation of consecutive Gegenbauer polynomials. The estimate is based on the recursion formula and a useful inequality of Gegenbauer polynomials. It is well known that for $0<\nu< 1$, $-1\le x\le 1$, one has
\begin{equation}\label{0 to 1}
    (1-x^2)^\frac{\nu}{2}|C_n^\nu(x)|<\frac{2^{1-\nu}}{\Gamma(\nu)}n^{\nu-1},
\end{equation}
where the constant $\frac{2^{1-\nu}}{\Gamma(\nu)}$ is optimal. (See Theorem 7.33.2 in \cite{Szego1975}). We believe that an analogous result of \eqref{0 to 1} exists for $\nu>1$, but now the following lemma, whose proof is left to Appendix \ref{appendix cancel}, is enough for our use. We will use $F_n^\nu$ instead of $C_n^\nu$  for the sake of notational consistency.
\begin{lemma}\label{nu>1}
For $\nu\ge 2$ and $-1\le x\le 1$, if $n\ge \max\{2\nu+2, 12\}$, then we have
   \begin{equation}\label{nnu>1}
    |(1-x^2){F}_n^\nu(x)|\le  \frac{\widetilde C_\nu}{n(n+2\nu)},
\end{equation}
where $\widetilde C_\nu$ is given in\eqref{Cnu}.
\end{lemma}

With the help of the above lemma, we can prove the following proposition.
\begin{proposition}\label{cancel}
Let $c_n^\nu=\max\limits_{0\le x\le 1}|{F}_{n+1}^\nu(x)-{F}_{n}^\nu(x)|$. For $\nu\ge 2$, we have
\begin{align*}
    c_n^\nu\le\frac{1}{n}\left(\frac{\widetilde C_{\nu}}{n+2\nu+1}+\widetilde C_{\nu+1}\right)
\end{align*}
if $n\ge \max\{2\nu+2, 12\}$.
\end{proposition}
\begin{proof}
   Recall the recursion formula for Gegenbauer polynomials
    \begin{equation*}
(1-x^2)2\nu C_{n}^{\nu+1}=-(n+1)xC_{n+1}^\nu+(n+2\nu)C_{n}^\nu,
    \end{equation*}
which, in view of \eqref{Fknu}, can be rewritten as
\begin{equation*}
    (1-x^2)(n+2\nu+1)F_{n}^{\nu+1}=-xF_{n+1}^\nu+F_{n}^\nu.
\end{equation*}
    Then by \eqref{nnu>1},
\begin{align}\label{difference}
    |F_{n+1}^\nu(x)-F_{n}^\nu(x)|&=|(1-x)F_{n+1}^\nu(x)-(1-x^2)(n+2\nu+1)F_{n}^{\nu+1}(x)|\notag\\
    &\le |(1-x^2)F_{n+1}^\nu(x)|+(n+2\nu+1)|(1-x^2)F_{n}^{\nu+1}(x)|\notag\\
    &\le \frac{\widetilde C_\nu}{(n+1)(n+2\nu+1)}+\frac{(n+2\nu+1)\widetilde C_{\nu+1}}{n(n+2\nu+2)}\notag\\
    &\le \frac{1}{n}\left(\frac{\widetilde C_{\nu}}{n+2\nu+1}+\widetilde C_{\nu+1}\right).\notag
\end{align}
\end{proof}

Recall that $\Tilde{F}_{n}'=F_{n-1}^\frac{7}{2}$, so we have
\begin{corollary}\label{cn}
Let $c_n=\max\limits_{0\le x\le 1}|\Tilde{F}_{n+1}'-\Tilde{F}_{n}'|$, then $c_n\le 0.12$ if $6\le n\le 29$ and
$c_n<0.026$ if $n\ge 30$.
\end{corollary}
\begin{proof}
    Direct computation by Matlab shows that the first assertion holds, and $c_n<0.026$ for $30\le n\le 428$ (the computational results are recorded in a supplemental data file). For $n>428$, by \eqref{Cnu}, we have
    $\widetilde C_\frac{7}{2}\le 9.19$ and
    $\widetilde C_\frac{9}{2}\le 11.02$, so we  can also deduce that
    \begin{equation*}
        c_n=c_{n-1}^\frac{7}{2}\le \frac{11.1}{n-1}<0.026.
    \end{equation*}
\end{proof}

\section{proof of main theorem for $\mathbb{S}^6$}\label{proofmain}
In this section, we will prove Theorem \ref{main} for $\mathbb{S}^6$ by induction argument, with the help of refined estimates on $b_k$'s.

We claim that $\beta=0$, which yields that $(1-x^2)^2G$ is a linear function by \eqref{G''}. Since $G$ is bounded on $(-1,1)$, we get $G\equiv 0$ and we are done.

So it suffices to show that $\beta=0$. We will argue by contradiction. If $\beta\neq 0$, then $0<\beta<\frac{1}{\alpha}$ since $a=\int_{-1}^1 (1-x^2) g=\frac{6}{7}(1-\alpha\beta)>0$. It then suffices to show $a=0$. We will achieve this by proving
 \begin{equation}
 \label{ainduction}
 a=\frac{6}{7}(1-\alpha\beta)\leq \frac{d_0}{\lambda_n},\ \forall n\geq 5 \text{ with }n\equiv 1\text{ (mod }4),
 \end{equation}
where $d_0=16$.

 As in \cite{GW2000} and \cite{LWY2022}, we will prove (\ref{ainduction}) by induction.

To begin with, we introduce the quantity
\begin{equation}
D=\sum\limits_{k=3}^{\infty}\left[\lambda_k(\lambda_k+4)(\lambda_k+6)-(14-\frac{74}{9\alpha})(\lambda_k+4)(\lambda_k+6)-\frac{160}{\alpha}\lambda_k-\frac{720}{\alpha}\right]b_{k}^{2}.
\end{equation}

Then by \eqref{G''} and \eqref{floor}, we get
\begin{align}
D
=&\lfloor G\rfloor^2-(14-\frac{74}{9\alpha})\int_{-1}^{1}|[(1-x^2)^2 G]''|^2-\frac{160}{\alpha}\int_{-1}^{1}(1-x^2)^3(G')^2\nonumber\\
&-\frac{720}{\alpha}\int_{-1}^{1}(1-x^2)^2G^2+\frac{16}{105}(\frac{2080}{3\alpha}+960)\beta^2\nonumber\\
\leq&(\frac{74}{9\alpha}-29)\int_{-1}^{1}|[(1-x^2)^2 G]''|^2+\frac{16}{105}(\frac{2080}{3\alpha}+960)\beta^2\nonumber\\
=&\frac{256}{35}(\frac{74}{9\alpha}-29)(7-\frac{1}{\alpha})\beta+\frac{512}{7}(\frac{13}{9\alpha}+2)\beta^2.\label{Dleq}
\end{align}

Since $D\geq 0$, $\alpha\geq \frac{1}{2}$ and $0<\beta<\frac{1}{\alpha}$, we obtain
\begin{equation}\label{beta}
\beta\geq \frac{9}{440}(29-\frac{74}{9\alpha})(7-\frac{1}{\alpha})\geq \frac{113}{88},
\end{equation}
and
\begin{equation}
\frac{256}{35}(\frac{74}{9\alpha}-29)(7-\frac{1}{\alpha})+\frac{512}{7}(\frac{13}{9\alpha}+2)\frac{1}{\alpha}\geq 0,
\end{equation}
which implies that
\begin{equation}\label{alpha1}
    \alpha<0.578.
\end{equation}

On the other hand, fix any integer $n\geq 3$, we have
\begin{align}
D
=& \sum\limits_{k=3}^{\infty}\left[\lambda_k(\lambda_k+4)(\lambda_k+6)-(14-\frac{74}{9\alpha})(\lambda_k+4)(\lambda_k+6)-\frac{160}{\alpha}\lambda_k-\frac{720}{\alpha}\right]b_{k}^{2}\nonumber\\
\geq& \sum\limits_{k=n+1}^{\infty}\left[\lambda_{n+1}-14+\frac{74}{9\alpha}-\frac{160\lambda_{n+1}+720}{(\lambda_{n+1}+4)(\lambda_{n+1}+6)\alpha}\right](\lambda_k+4)(\lambda_{k}+6)b_{k}^2\nonumber\\
&+\sum\limits_{k=3}^{n}\left[\lambda_{k}-14+\frac{74}{9\alpha}-\frac{160\lambda_{k}+720}{(\lambda_{k}+4)(\lambda_{k}+6)\alpha}\right](\lambda_k+4)(\lambda_{k}+6)b_{k}^2\nonumber\\
\geq &(\lambda_{n+1}-14+\frac{275}{63\alpha})\sum\limits_{k=n+1}^{\infty}(\lambda_k+4)(\lambda_{k}+6)b_{k}^2\nonumber\\
&+\sum\limits_{k=3}^{n}(\lambda_k-14+\frac{176}{63}\alpha)(\lambda_k+4)(\lambda_{k}+6)b_{k}^2\nonumber\\
=&\sum\limits_{k=3}^{n}(\lambda_k-\lambda_{n+1}-\frac{11}{7\alpha})(\lambda_k+4)(\lambda_{k}+6)b_{k}^2\nonumber\\
&+(\lambda_{n+1}-14+\frac{275}{63\alpha})\left[\frac{256}{35}(7-\frac{1}{\alpha})\beta-\frac{128}{7}\beta^2-360b_{2}^{2}\right].\label{split}
\end{align}

Combining \eqref{Dleq} and \eqref{split}, we get
\begin{align}\label{408}
0
\leq&\frac{256}{35}(7-\frac{1}{\alpha})(\frac{27}{7\alpha}-15-\lambda_{n+1})\beta+\frac{128}{7}(\lambda_{n+1}-6+\frac{71}{7\alpha})\beta^2\nonumber\\
&+\frac{176}{63\alpha}(\lambda_2+4)(\lambda_2+6)b_{2}^{2}+\sum\limits_{k=2}^{n}(\lambda_{n+1}-\lambda_k+\frac{11}{7\alpha})(\lambda_k+4)(\lambda_{k}+6)b_{k}^2.
\end{align}

Then we can start the induction procedure to prove $a\leq \frac{16}{\lambda_n}$, for all  $n\geq 5$ with $n\equiv 1\text{ (mod }4)$. Note that from \eqref{beta} and \eqref{alpha1}, we already have $a\leq 0.221\leq\frac{16}{\lambda_5}$.

By induction, now we assume $a\leq \frac{16}{\lambda_n}$ for some $n\geq 5$ with $n\equiv 1\text{ (mod }4)$. Then we will show that $a\leq \frac{16}{\lambda_{n+4}}$. We argue by contradiction and suppose $a> \frac{16}{\lambda_{n+4}}$ on the contrary.

Let $B_k=\frac{9\alpha^2}{32}(\lambda_{n+1}-\lambda_k+\frac{11}{7\alpha})(2k+5)$, then for every even $k$, we have
\begin{align*}
&\frac{9\alpha^2}{32}\left[(\lambda_{n+1}-\lambda_k+\frac{11}{7\alpha})(\lambda_k+4)(\lambda_k+6)b_{k}^{2}+(\lambda_{n+1}-\lambda_{k+1}+\frac{11}{7\alpha})(\lambda_{k+1}+4)(\lambda_{k+1}+6)b_{k+1}^{2}\right]\\
=&B_k(\int_{-1}^{1}(1-x^2)\Tilde{F}_{k}'g)^2+B_{k+1}(\int_{-1}^{1}(1-x^2)\Tilde{F}_{k+1}'g)^2\\
=&B_k\left[(\int_{0}^{1}(1-x^2)\Tilde{F}_{k}'g)^2+(\int_{-1}^{0}(1-x^2)\Tilde{F}_{k}'g)^2\right]+B_{k+1}\left[(\int_{0}^{1}(1-x^2)\Tilde{F}_{k+1}'g)^2+(\int_{-1}^{0}(1-x^2)\Tilde{F}_{k+1}'g)^2\right]\\
+&2B_k\int_{0}^{1}(1-x^2)\Tilde{F}_{k}'g\int_{-1}^{0}(1-x^2)(\Tilde{F}_{k}'+\Tilde{F}_{k+1}')g
+2B_{k+1}\int_{0}^{1}(1-x^2)(\Tilde{F}_{k+1}'-\Tilde{F}_{k}')g\int_{-1}^{0}(1-x^2)\Tilde{F}_{k+1}'g\\
+&2(B_{k+1}-B_{k})\int_{0}^{1}(1-x^2)\Tilde{F}_{k}'g\int_{-1}^{0}(1-x^2)\Tilde{F}_{k+1}'g\\
:=&R_{k,1}+R_{k,2}+R_{k,3}.
\end{align*}

%where $R_k$ is the remainder term that can be written as three small terms
%\begin{align*}R_k:=
%&B_k\int_{0}^{1}(1-x^2)\Tilde{F}_{k}'g\int_{-1}^{0}(1-x^2)\Tilde{F}_{k}'g
%+B_{k+1}\int_{0}^{1}(1-x^2)\Tilde{F}_{k+1}'g\int_{-1}^{0}(1-x^2)\Tilde{F}_{k+1}'g\\
%=&(\lambda_{n+1}-\lambda_{k}+\frac{11}{7\alpha})(2k+5)\int_{0}^{1}(1-x^2)\Tilde{F}_{k}'g\int_{-1}^{0}(1-x^2)(\Tilde{F}_{k}'+\Tilde{F}_{k+1}')g\\
%-&(\lambda_{n+1}-\lambda_{k}+\frac{11}{7\alpha})(2k+5)\int_{0}^{1}(1-x^2)\Tilde{F}_{k}'g\int_{-1}^{0}(1-x^2)\Tilde{F}_{k+1}'g\\
%+&(\lambda_{n+1}-\lambda_{k+1}+\frac{11}{7\alpha})(2k+7)\int_{0}^{1}(1-x^2)(\Tilde{F}_{k+1}'-\Tilde{F}_{k}')g\int_{-1}^{0}(1-x^2)\Tilde{F}_{k+1}'g\\
%+&(\lambda_{n+1}-\lambda_{k+1}+\frac{11}{7\alpha})(2k+7)\int_{0}^{1}(1-x^2)\Tilde{F}_{k}'g\int_{-1}^{0}(1-x^2)\Tilde{F}_{k+1}'g\\
%=&B_k\int_{0}^{1}(1-x^2)\Tilde{F}_{k}'g\int_{-1}^{0}(1-x^2)(\Tilde{F}_{k}'+\Tilde{F}_{k+1}')g
%+B_{k+1}\int_{0}^{1}(1-x^2)(\Tilde{F}_{k+1}'-\Tilde{F}_{k}')g\int_{-1}^{0}(1-x^2)\Tilde{F}_{k+1}'g\\
%+&(2\lambda_{n+1}+\frac{22}{7\alpha}-6k^2-36k-42)\int_{0}^{1}(1-x^2)\Tilde{F}_{k}'g\int_{-1}^{0}(1-x^2)\Tilde{F}_{k+1}'g.
%\end{align*}

Recall the definition of $\mathcal{A}_{k}^{\pm}$ from Theorem \ref{bk}. By Theorem \ref{bk}, we have
\begin{align}\label{R1}
    R_{k,1}&=B_k\left[(\int_{0}^{1}(1-x^2)\Tilde{F}_{k}'g)^2+(\int_{-1}^{0}(1-x^2)\Tilde{F}_{k}'g)^2\right]\notag\\
    &+B_{k+1}\left[(\int_{0}^{1}(1-x^2)\Tilde{F}_{k+1}'g)^2+(\int_{-1}^{0}(1-x^2)\Tilde{F}_{k+1}'g)^2\right]\notag\\
    &\le B_k\left(|\mathcal{A}_{k}^{+}|^2+|\mathcal{A}_{k}^{-}|^2\right)+B_{k+1}\left(|\mathcal{A}_{k+1}^{+}|^2+|\mathcal{A}_{k+1}^{-}|^2\right).
\end{align}

Let $c_k$ be defined as in Corollary \ref{cn}, then we have
\begin{align*}
&|\int_{-1}^{0}(1-x^2)(\Tilde{F}_{k}'+\Tilde{F}_{k+1}')g|\leq c_ka_-=c_k(1-\lambda)a,\\
&|\int_{0}^{1}(1-x^2)(\Tilde{F}_{k+1}'-\Tilde{F}_{k}')g|\leq c_ka_{+}=c_k\lambda a.
\end{align*}
So
\begin{align}\label{R2}
    R_{k,2}&=2B_k\int_{0}^{1}(1-x^2)\Tilde{F}_{k}'g\int_{-1}^{0}(1-x^2)(\Tilde{F}_{k}'+\Tilde{F}_{k+1}')g\notag\\
    &+2B_{k+1}\int_{0}^{1}(1-x^2)(\Tilde{F}_{k+1}'-\Tilde{F}_{k}')g\int_{-1}^{0}(1-x^2)\Tilde{F}_{k+1}'g\notag\\
&\le 2(B_k+B_{k+1})c_k\lambda(1-\lambda)a^2.
\end{align}

Finally by Lemma \ref{minFk}, we have
\begin{equation}\label{R3}
    R_{k,3}\le \left\{\begin{aligned}
        &2(B_{k+1}-B_k)\lambda(1-\lambda)a^2, &\text{ if } B_k\le B_{k+1},\\
        &2(B_{k}-B_{k+1})m_0(1-\lambda) a^2, &\text{ if } B_{k+1}<B_{k}.
    \end{aligned}\right.
\end{equation}

Now from \eqref{R1}, \eqref{R2} and \eqref{R3}, we can get the estimate of each term in the summation in \eqref{split} for each even $k$.
\begin{align}\label{Rs}
&\frac{9\alpha^2}{32}[(\lambda_{n+1}-\lambda_k+\frac{11}{7\alpha})(\lambda_k+4)(\lambda_k+6)b_{k}^{2}+(\lambda_{n+1}-\lambda_{k+1}+\frac{11}{7\alpha})(\lambda_{k+1}+4)(\lambda_{k+1}+6)b_{k+1}^{2}]\notag\\
\leq
&B_k\left(|\mathcal{A}_{k}^{+}|^2+|\mathcal{A}_{k}^{-}|^2\right)+B_{k+1}\left(|\mathcal{A}_{k+1}^{+}|^2+|\mathcal{A}_{k+1}^{-}|^2\right)+2(B_k+B_{k+1})c_k\lambda(1-\lambda)a^2\notag\\
+&\left\{\begin{aligned}
        &2(B_{k+1}-B_k)\lambda(1-\lambda)a^2, &\text{ if } B_k\le B_{k+1},\\
        &2(B_{k}-B_{k+1})m_0(1-\lambda) a^2, &\text{ if } B_{k+1}<B_{k}.
    \end{aligned}\right.
\end{align}
\begin{remark}
Note that this estimate is better than the one in $\mathbb{S}^4$ case. Cancellation of consecutive Gegenbauer polynomials is used in the proof.
\end{remark}
The right hand side above can be viewed as a function $f_{k,a}(\lambda)$ of $\lambda=\frac{a_+}{a}$. The following Proposition yields that the worst case is $\lambda=1$. In particular, in this case, we can drop the small terms $R_{k,2}$ and $R_{k,3}$. The proof is left to Appendix \ref{appendix lambda==1}.

\begin{proposition}\label{lambda=1}
Suppose $a$ satisfies $\frac{d_0}{\lambda_{n+4}}\leq a\leq\frac{d_0}{\lambda_n}$ for some $n\geq 5$ with $n\equiv 1\text{ (mod }4)$ where $d_0=16$. Let $f_{k,a}(\lambda)$ be defined as above. Then for any $k$ even, we have for $n\geq 41$,\\
(1) If $\lambda_k\leq \frac{1}{4}\lambda_n$, then
\begin{equation}
f_{k,a}(\lambda)\leq f_{k,a}(1)=
B_k(a-\frac{1-b}{d}\lambda_k a^2)^2+B_{k+1}(a-\frac{1-b}{d}\lambda_{k+1} a^2)^2.
\end{equation}
(2) If $\frac{1}{4}\lambda_n<\lambda_k\leq \lambda_n$, then
\begin{equation}
f_{k,a}(\lambda)\leq f_{k,a}(1)=
B_k(ba+(1-b)\frac{d}{4\lambda_k})^2+B_{k+1}(ba+(1-b)\frac{d}{4\lambda_{k+1}})^2.
\end{equation}
For $5\leq n\leq 65,$ we have

(1) If $\lambda_k\leq \frac{1}{4}\lambda_n$, then
\begin{equation}
f_{k,a}(\lambda)\leq
B_k(a-\frac{1-b}{d}\lambda_k a^2)^2+B_{k+1}(a-\frac{1-b}{d}\lambda_{k+1} a^2)^2+\frac{1}{2}(B_k+B_{k+1})c_k a^2.
\end{equation}
(2) If $\frac{1}{4}\lambda_n<\lambda_k\leq \lambda_n$, then
\begin{equation}
f_{k,a}(\lambda)\leq
B_k(ba+(1-b)\frac{d}{4\lambda_k})^2+B_{k+1}(ba+(1-b)\frac{d}{4\lambda_{k+1}})^2+\frac{1}{2}(B_k+B_{k+1})c_k a^2.
\end{equation}
\end{proposition}

In the following, we will assume $n>10000$. The case when $n<10000$ is checked by Matlab and is left to Appendix \ref{proofsmall} .

With the help of Proposition \ref{lambda=1} and by plugging it into \eqref{408}, we obtain
\begin{align}
0\leq&\frac{256}{35}(7-\frac{1}{\alpha})(\frac{27}{7\alpha}-15-\lambda_{n+1})\frac{1}{\alpha}(1-\frac{7}{6}a)+\frac{128}{7}(\lambda_{n+1}-6+\frac{71}{7\alpha})\frac{1}{\alpha^2}(1-\frac{7}{6}a)^2\nonumber\\
+&\frac{176}{63}\alpha(\lambda_2+4)(\lambda_2+6)b_{2}^{2}\nonumber\\
+&\frac{32}{9\alpha^2}\sum_{k=2}^{\frac{n-3}{2}}(\lambda_{n+1}-\lambda_{k}+\frac{11}{7\alpha})(2k+5)(1-\frac{1-b}{d}\lambda_{k}\frac{16}{\lambda_{n+4}})^2a^2\nonumber\\
+&\frac{32}{9\alpha^2}\sum_{k=\frac{n-1}{2}}^{n}(\lambda_{n+1}-\lambda_{k}+\frac{11}{7\alpha})(2k+5)(ba+(1-b)\frac{d}{4\lambda_{k}})^2.\nonumber\\
\leq &-\frac{512}{7}(\lambda_{n+1}+\frac{51}{7})(1-\frac{7}{6}a)+\frac{512}{7}(\lambda_{n+1}+\frac{100}{7})(1-\frac{7}{6}a)^2+\frac{22528}{63\alpha}a^2\nonumber\\
+&\frac{128}{9}\sum_{k=2}^{\frac{n-3}{2}}(\lambda_{n+1}-\lambda_{k}+\frac{22}{7})(2k+5)(1-\frac{1-b}{d}\lambda_{k}\frac{16}{\lambda_{n+4}})^2a^2\nonumber\\
+&\frac{128}{9}\sum_{k=\frac{n-1}{2}}^{n}(\lambda_{n+1}-\lambda_{k}+\frac{22}{7})(2k+5)(ba+(1-b)\frac{d}{4\lambda_{k}})^2.\nonumber\\
=:&g_{n,1}(a)+g_{n,2}(a)+g_{n,3}(a)=g_n(a)
 \end{align}
where $g_{n,1}, g_{n,2}$ and $ g_{n,3}$ are defined at the last equality.

For $g_{n,2}(a)$, we can decompose it into three summations
\begin{align}\label{firstsum}
g_{n,2}(a)=\frac{128}{9}\left[S_1-\frac{34(1-b)}{d\lambda_{n+4}}S_2+\frac{289(1-b)^2}{d^2\lambda_{n+4}^{2}}S_3\right]a^2,
\end{align}

where
\begin{align}\label{S1}
S_1=\sum_{k=2}^{\frac{n-3}{2}}(\lambda_{n+1}-\lambda_{k}+\frac{11}{7\alpha})(2k+5)=\frac{7}{32}n^4+\frac{23}{8}n^3-\frac{115}{112}n^2-\frac{4265}{56}n-\frac{20075}{224},
\end{align}
\begin{align}\label{S2}
S_2
&=\sum_{k=2}^{\frac{n-3}{2}}(\lambda_{n+1}-\lambda_{k}+\frac{11}{7\alpha})(2k+5)\lambda_{k}\nonumber\\
&=\frac{5}{192}n^6+\frac{1}{2}n^5+\frac{3611}{1344}n^4-\frac{9}{28}n^3-\frac{100207}{5376}n^2-\frac{1393237}{896}n-\frac{1040985}{1024},
\end{align}
\begin{align}\label{S3}
S_3
&=\sum_{k=2}^{\frac{n-3}{2}}(\lambda_{n+1}-\lambda_{k}+\frac{11}{7\alpha})(2k+5)\lambda_{k}^{2}\nonumber\\
&=\frac{13 }{3072}n^8+\frac{41 }{384}n^7+\frac{1525 }{1792}n^6+\frac{3011}{2688}n^5-\frac{48697}{3584}n^4-\frac{14917 }{384}n^3\nonumber\\
&+\frac{1000525 }{5376}n^2-\frac{1393237 }{896}n-\frac{1040985}{1024}
\end{align}

For $g_{n,3}(a)$, direct calculation yields that
\begin{align}\label{secondsum}
&\sum_{k=\frac{n-1}{2}}^{n}(\lambda_{n+1}-\lambda_k+\frac{22}{7})(2k+5)(ba+(1-b)\frac{d}{4\lambda_k})^2\nonumber\\
=&b^2S_4 a^2+2b(1-b)(\lambda_{n+1}+\frac{22}{7})\frac{d}{4}S_5 a-2b(1-b)\frac{d}{4}S_6 a+(1-b)^2\frac{d^2}{16}(\lambda_{n+1}+\frac{22}{7})S_7-(1-b)^2\frac{d^2}{16}S_5,
\end{align}
where
\begin{align}\label{S4}
S_4
&=\sum_{k=\frac{n-1}{2}}^{n}(\lambda_{n+1}-\lambda_k+\frac{22}{7})(2k+5)=\frac{9}{32}n^4+\frac{33}{8}n^3+\frac{2763}{112}n^2+\frac{3753}{56}n+\frac{15147}{224},
\end{align}
\begin{equation}\label{S5}
S_5=\sum_{k=\frac{n-1}{2}}^{n}\frac{2k+5}{\lambda_k}=\sum_{k=\frac{n-1}{2}}^{n}(\frac{1}{k}+\frac{1}{k+5})\geq 1.3863,
\end{equation}
\begin{equation}\label{S6}
S_6=\sum_{k=\frac{n-1}{2}}^{n}(2k+5)=\frac{3}{4}n^2+\frac{9}{2}n+\frac{27}{4},
\end{equation}
\begin{align}\label{S7}
S_7
&=\sum_{k=\frac{n-1}{2}}^{n}\frac{2k+5}{\lambda_{k}^{2}}=\frac{1}{5}\sum_{k=\frac{n-1}{2}}^{n}(\frac{1}{k^2}-\frac{1}{(k+5)^2})\nonumber\\
&=\frac{1}{5}\left(\frac{3}{(n+1)^2}-\frac{1}{(n+2)^2}+\frac{3}{(n+3)^2}-\frac{1}{(n+4)^2}+\frac{3}{(n+5)^2}+\frac{4}{(n+7)^2}+\frac{4}{(n-1)^2}\right)\nonumber\\
&\leq \frac{3}{n^2}.
\end{align}

To get a contradiction, we need to show that $g_n(a)$ is negative for $\frac{16}{\lambda_{n+4}}<a<\frac{16}{\lambda_n}$. Direct computation gives that for $n>10000$ with $n\equiv 1\text{ (mod }4)$, we have the following three estimates

\begin{align*}
g_{n,1}(a)=&-\frac{512}{7}(\lambda_{n+1}+\frac{51}{7})(1-\frac{7}{6}a)+\frac{512}{7}(\lambda_{n+1}+\frac{100}{7})(1-\frac{7}{6}a)^2+\frac{22528}{63\alpha}a^2\\
=&\frac{512}{7}\left[7-\frac{7}{6}a\lambda_{n+1}-\frac{149}{6}a+\frac{49}{36}\lambda_{n+1}a^2+\frac{175}{9}a^2\right]+\frac{22528}{63\alpha}a^2\\
\leq&\frac{512}{7}\left(7-\frac{56}{3}\frac{\lambda_{n+1}}{\lambda_{n+4}}-\frac{1192}{3\lambda_{n+4}}+\frac{3136}{9\lambda_{n}}+\frac{44800}{9\lambda{n}^2}\right)+\frac{91543}{\lambda_{n}^2}\leq -853.33,
\end{align*}

\begin{align*}
g_{n,2}(a)=& \frac{128}{9}\left[S_1-\frac{67}{25\lambda_{n+4}}S_2+\frac{4489}{2500\lambda_{n+4}^2}S_3\right]a^2\\
\leq&\frac{128}{9}\left[\left(\frac{56n^4}{\lambda_{n}^2}+\frac{736n^3}{\lambda_{n}^2}+\frac{1840n^2}{\lambda_{n+4}^{2}}-\frac{136480n}{7\lambda_{n+4}^{2}}-\frac{160600}{7\lambda_{n+4}^2}\right)\right.\\
+&\left.\left(-\frac{268n^6}{15\lambda_{n+4}^3}-\frac{343n^5}{\lambda_{n+4}^{3}}-\frac{1843n^4}{\lambda_{n+4}^3}+\frac{221n^3}{\lambda_{n}^3}+\frac{51154n^2}{\lambda_{n}^3}+\frac{231234n}{\lambda_{n}^3}+\frac{40683}{\lambda_{n}^3}\right)\right.\\
+&\left.\left(\frac{1.94524n^8}{\lambda_{n}^4}+\frac{49.0797n^7}{\lambda_{n}^4}+\frac{391.18n^6}{\lambda_{n}^4}+\frac{515n^5}{\lambda_{n+4}^4}-\frac{6245n^4}{\lambda_{n+4}^4}-\frac{17856n^3}{\lambda_{n+4}^4}\right.\right.\\
-&\left.\left.\frac{85550n^2}{\lambda_{n}^4}-\frac{714770n}{\lambda_{n}^4}-\frac{467298}{\lambda_{n+4}^4}\right)\right]\\
\leq& 571.123,
\end{align*}

\begin{align*}
g_{n,3}(a)=&\frac{128}{9}\left[0.1089S_4a^2+\frac{2211}{2500}(\lambda_{n+1}+\frac{22}{7})S_5 a-\frac{2211}{2500}S_6 a+\frac{4489}{2500}(\lambda_{n+1}+\frac{22}{7})S_7-\frac{4489}{2500}S_5\right]\\
\leq&\frac{128}{9}\left[\left(\frac{7.8408n^4}{\lambda_{n}^2}+\frac{115n^3}{\lambda_{n}^2}+\frac{688n^2}{\lambda_{n}^2}+\frac{1869n}{\lambda_{n}^2}+\frac{1886}{\lambda_{n}^2}\right)+19.6166\frac{\lambda_{n+1}+\frac{22}{7}}{\lambda_n}-\frac{10.6128n^2}{\lambda_{n+4}}\right.\\
+&\left.\frac{13467}{2500}\frac{\lambda_{n+1}+\frac{22}{7}}{n^2}-2.48923\right]\\
\leq& 280.95.
\end{align*}

Combining three estimates above, we found
\begin{equation*}
0\leq g_n(a)\leq -853.33+571.123+280.95<-1.257<0,
\end{equation*}
for all $n>10000$ with $n\equiv 1\text{ (mod }4)$ and $\frac{16}{\lambda_{n+4}}<a\leq\frac{16}{\lambda_n}$, which is a contradiction. Consequently, we finish the proof of Theorem \ref{main}.

\appendix
\section{proof of Lemma \ref{A2-5}}\label{appendix A}
In this appendix, we prove Lemma~\ref{A2-5}.

\begin{proof}[Proof of Lemma~\ref{A2-5}]
Define $A_{m,n}^+=\int_0^1 x^m(1-x^2)^ng$, $A_{m,n}^-=\int_{-1}^0 |x|^m(1-x^2)^ng$, and $A_{m,n}=A_{m,n}^++A_{m,n}^-$. We begin with the estimate of $A_2$. By definition,
\begin{eqnarray*}
    |A_2|=|\int_{-1}^{1}x(1-x^2)g|\leq \max\left\{A_{1,1}^+,A_{1,1}^-\right\}.
\end{eqnarray*}
By Cauchy-Schwartz inequality and \eqref{N310},
\begin{align*}
a_+-A^+_{1,1}=\int_0^1 (1-x^2)(1-x)g\ge\frac{(\int_0^1(1-x^2)g)^2}{\int_0^1(1+x)g}\ge a_+^2,
\end{align*}
so
\begin{equation*}
    A^+_{1,1}\le a_+-a_+^2.
\end{equation*}
Similarly,
\begin{equation*}
    A^-_{1,1}\le a_--a_-^2.
\end{equation*}
Since  $a<1$ and we have assumed $\lambda\ge \frac{1}{2}$, we conclude that
\begin{eqnarray*}
    |A_2|\le a_+-a_+^2.
\end{eqnarray*}
The estimate of $|A_4|$ is similar to that of $|A_2|$. By definition,
\begin{eqnarray*}
    A_4=\int_{-1}^{1}(1-x^2)g\widetilde{F}_{4}'=\frac{1}{8}\int_{-1}^{1}(1-x^2)(11x^2-3)xg=A_{1,1}-\frac{11}{8}A_{1,2}.
\end{eqnarray*}
By Cauchy-Schwartz inequality and \eqref{309},
\begin{align*}
    A_{1,2}\ge \frac{(A_{1,1}^+)^2}{\int_0^1 xg}\ge 2(A_{1,1}^+)^2,
\end{align*}
so
\begin{equation*}
    A_4^+\le A_{1,1}^+-\frac{11}{4}(A_{1,1}^+)^2
\end{equation*}
On the other hand,
\begin{align*}
    A_4^+\ge \frac{1}{8}\min_{0\le x\le1}\{(11x^2-3)x\}\int_{0}^{1}(1-x^2)g=-\frac{1}{4\sqrt{11}}a_+.
\end{align*}
In the same way,
\begin{eqnarray*}
   -(A_{1,1}^--\frac{11}{4}(A_{1,1}^-)^2) \le A_4^-\le \frac{1}{4\sqrt{11}}a_-.
\end{eqnarray*}
Since $\lambda\ge \frac{1}{2}$, we conclude that
\begin{eqnarray*}
    |A_4|\le  A_{1,1}^+-\frac{11}{4}(A_{1,1}^+)^2+\frac{1}{4\sqrt{11}}a_-\le (a_+-a_+^2)-\frac{11}{4} (a_+-a_+^2)^2+\frac{1}{4\sqrt{11}}a_-.
\end{eqnarray*}
The estimates of $A_3$ and $A_5$ are slightly different. For $A_3$, we write
\begin{eqnarray*}
    A_3=\int_{-1}^{1}(1-x^2)g\widetilde{F}_{3}'=\frac{1}{8}\int_{-1}^{1}(1-x^2)(9x^2-1)g=\frac{1}{8}(9A_{2,1}-a).
\end{eqnarray*}
By Cauchy-Schwartz inequality and \eqref{308},
\begin{align*}
    (A_{2,1}^+)^2&\le \int_0^1(1-x^2)^2g \int_0^1 x^4g\\
    &\le (a_+-A_{2,1}^+)(\frac{a+1}{2}-a_+-A_{2,1}^+),
\end{align*}
so \begin{equation}\label{A01}
    A_{2,1}^+\le a_+-\frac{2a_+^2}{a+1}.
\end{equation}
In the same way,
\begin{equation*}
    A_{2,1}^-\le a_--\frac{2a_-^2}{a+1}.
\end{equation*}
Hence,
\begin{equation*}
    A_{2,1}\le a-\frac{2a_+^2+2a_-^2}{a+1}=a-\frac{2a^2}{a+1}(2\lambda^2-2\lambda+1).
\end{equation*}
Therefore
\begin{eqnarray*}
    A_3\le a-\frac{9}{4}\frac{a^2}{a+1}(2\lambda^2-2\lambda+1),
\end{eqnarray*}
which, together with the definition of $A_3$, implies
\begin{eqnarray*}
    |A_3|\le \max\left\{ a-\frac{9}{4}\frac{a^2}{a+1}(2\lambda^2-2\lambda+1), \frac{a}{8}\right\}= a-\frac{9}{4}\frac{a^2}{a+1}(2\lambda^2-2\lambda+1).
\end{eqnarray*}

Finally, for $A_5$, we have
\begin{align*}
    A_5=\frac{1}{80}\int_{-1}^1(1-x^2)(3-66x^2+143x^4)g=\frac{1}{80}(80a-143A_{2,2}-77A_{2,0}).
\end{align*}

By Cauchy-Schwartz inequality and \eqref{308},
\begin{align*}
    A_{2,2}^+\ge \frac{(A_{2,1}^+)^2}{\int_0^1 x^2g}\ge \frac{ (A_{2,1}^+)^2}{\frac{a+1}{2}-a_+},
\end{align*}
so by \eqref{A01},
\begin{align*}
    A_5^+&\le \frac{1}{80}\big(80a_+-\frac{ 143(A_{2,1}^+)^2}{\frac{a+1}{2}-a_+}-77(a_+-A_{2,1}^+)\big)\\
    &\le \frac{1}{80}\Big(3a_+-11(a_+-\frac{2a_+^2}{a+1})(\frac{26a_+}{a+1}-7)\Big)\\
    &=a_+-\frac{11a_+^2}{2(a+1)}+\frac{143a_+^3}{10(a+1)^2}.
\end{align*}
Therefore
\begin{align*}
    A_5\le a-\frac{11(a_+^2+a_-^2)}{2(a+1)}+\frac{143(a_+^3+a_-^3)}{10(a+1)^2}.
\end{align*}
On the other hand,
\begin{align*}
    A_5\ge \frac{1}{80}\min_{-1\le x\le1}\{3-66x^2+143x^4\}\int_{-1}^{1}(1-x^2)g=-\frac{3}{52}a.
\end{align*}
From \eqref{beta} and the estimates of $|A_2|$ and $|A_3|$, we can deduce that  $a<0.125$, so now it is not hard to see that
\begin{align*}
    |A_5|\le a-\frac{11(a_+^2+a_-^2)}{2(a+1)}+\frac{143(a_+^3+a_-^3)}{10(a+1)^2}.
\end{align*}
Thus the proof of Lemma \ref{A2-5} is completed.
\end{proof}

\section{proof of Lemma \ref{minFk} and  \ref{ptwise}}\label{appendix B}
In this appendix we prove Lemma \ref{minFk} and Lemma \ref{ptwise}. The proofs are technical and make use of many quantitative properties of Gegenbauer polynomials.

 Before we prove Lemma~\ref{minFk}, we first state some general lemma about Gegenbauer polynomials.  Denote by $x_{nk}(\nu)$, $k=1,\cdots ,n$, the zeros of $C_n^\nu(x)$ enumerated in decreasing order, that is, $1>x_{n1}(\nu)>\cdots>x_{nn}(\nu)>-1$.

\begin{lemma}[Corollary 2.3 in Area et al.\cite{Area2004}]\label{xn1}
    For any $n \ge 2$ and for every $\nu\ge 1$, the inequality
\begin{align}
    x_{n1}(\nu)\le \sqrt{\frac{(n-1)(n+2\nu-2)}{(n+\nu-2)(n+\nu-1)}}\cos(\frac{\pi}{n+1})
\end{align}
    holds.
\end{lemma}

The next lemma is well-known and it is valid for many other orthogonal polynomials.
\begin{lemma}[Olver et al. \cite{Olver2010}]\label{localmax}
Denote by $y_{nk}(\nu)$, $k=0,1,\cdots ,n-1,n$, the local maxima of $|C_n^\nu(x)|$ enumerated in decreasing order, then $y_{n0}(\nu)=1, y_{nn}(\nu)=-1$, and we have
\begin{enumerate}
\item[$(a)$]\quad
   $y_{nk}(\nu)=x_{n-1,k}(\nu+1),\ k=1,\cdots, n-1.$
\item [$(b)$]\quad
$|C_n^\nu(y_{n0}(\nu))|>|C_n^\nu(y_{n1}(\nu))|>\cdots>|C_n^\nu(y_{n,[\frac{n+1}{2}]}(\nu))|.$
\item[$(c)$]\quad
$(C_n^\nu)^{(k)}(x)>0$ on $(x_{n1}(\nu),1)$ for all $k=0,1,\cdots, n.$
\end{enumerate}
\end{lemma}

\begin{proof}[Proof of Lemma~\ref{minFk}]
 Direct computation by Matlab shows that the lemma holds for $8\le k\le 200$, so in what follows we may assume $k>200$. . By Lemma~\ref{xn1} and \eqref{201}, we know that the minimum of $\widetilde{F}_k'$ on $[0,1]$ is achieved at the point
\begin{align}\label{minpt}
    x_{k-2,1}(\frac{9}{2})\le \sqrt{\frac{(k-3)(k+5)}{(k+\frac{3}{2})(k+\frac{1}{2})}}\cos(\frac{\pi}{k-1})< 1-\frac{12.5}{k^2}.
\end{align}

 Taking $N=4$ in Lemma~\ref{lem33}, we obtain
    \begin{align}\label{**}
     \Tilde{F}_k'(\cos{\zeta})
     =F_{k-1}^{\frac{7}{2}}(\cos{\zeta})
     &=48\sqrt{\frac{2}{\pi}}\left(\sum_{m=0}^{3}t_m(3)\frac{\Gamma(k)}{\Gamma(k+m+\frac{7}{2})}\frac{\cos{(\delta_{k-1,m})}}{\sin^{m+\frac{7}{2}}{\zeta}}+\widetilde{R}\right),
\end{align}
where $\widetilde{R}$ satisfies
\begin{align}\label{R}
    |\widetilde{R}|\le t_4(3) \frac{\Gamma(k)}{\Gamma(k+\frac{15}{2})}(\sin{\zeta})^{-\frac{15}{2}}\cdot
\begin{cases}
\sec{\zeta}&\text{ if }0<\zeta\leq\frac{\pi}{4},\\
2\sin{\zeta}&\text{ if }\frac{\pi}{4}<\zeta<\frac{\pi}{2},
\end{cases}
\end{align}
the value of $t_m(3)$ for $0\le m\le 3$ are listed below:
$$t_0(3)=1,t_1(3)=\frac{35}{8}, t_2(3)=\frac{945}{128}, t_3(3)=\frac{3465}{1024}, t_4(3)= -\frac{45045}{32768}.$$

Let $\sin{\zeta}=\frac{l}{k}$. Then by \eqref{minpt} we can assume $l\ge 5$. From \eqref{R} we know that if $l\le \frac{k}{\sqrt{2}}$, then
\begin{align}\label{tilde R1}
    |\widetilde{R}|\le |t_4(3)|\frac{k^\frac{15}{2}\Gamma(k)}{l^\frac{15}{2}\Gamma(k+\frac{15}{2})}\frac{1}{\sqrt{1-\frac{l^2}{k^2}}}
    < \frac{1.5}{l^\frac{15}{2}\sqrt{1-\frac{l^2}{k^2}}};
\end{align}
while if $l> \frac{k}{\sqrt{2}}$, then
\begin{align}\label{tiled R2}
    |\widetilde{R}|\le 2|t_4(3)|\frac{k^\frac{15}{2}\Gamma(k)}{l^\frac{15}{2}\Gamma(k+\frac{9}{2})}
    <  \frac{3(\sqrt{2})^\frac{15}{2}}{k^\frac{15}{2}}.
\end{align}
To get the desired lower bound, we shall use the following simple estimates.
\begin{equation}\label{cos}
    \cos(x+\delta)=\cos x-\delta \sin(x+h\delta)\ge \cos x-|\delta|.
\end{equation}
\begin{equation}\label{sin}
    \zeta-\sin\zeta\le (\frac{\pi}{2}-1) \sin^3{\zeta} \le \sin^3{\zeta}, \ 0<\zeta<\frac{\pi}{2}.
\end{equation}
 With the help of \eqref{cos} and \eqref{sin}, we have
 \begin{align}
   \cos{(\delta_{k-1,m})}
    &=\cos{\left((k+\frac{5}{2}+m)\zeta-(\frac{7}{2}-m)\frac{\pi}{2}\right)}\notag\\
    &=\cos\left((k+\frac{5}{2}+m)\frac{l}{k}+(k+\frac{5}{2}+m)(\zeta-\sin\zeta)-(\frac{7}{2}-m)\frac{\pi}{2}\right)\notag\\
    &\ge \cos\left(l-(\frac{7}{2}-m)\frac{\pi}{2}\right)-\left((k+\frac{5}{2}+m)(\zeta-\sin\zeta)+(\frac{5}{2}+m)\frac{l}{k}\right)\notag\\
    &\ge \cos\left(l-(\frac{7}{2}-m)\frac{\pi}{2}\right)-(3+m)\frac{l}{k}.\label{cosdelta}
\end{align}
Therefore we have
\begin{align*}
 &\sum_{m=0}^{3}t_m(3)\frac{\Gamma(k)}{\Gamma(k+m+\frac{7}{2})}\frac{\cos{(\delta_{k-1,m})}}{\sin^{m+\frac{7}{2}}{\zeta}}=\sum_{m=0}^{3}t_m(3)\frac{k^{m+\frac{7}{2}}\Gamma(k)}{\Gamma(k+m+\frac{7}{2})}\frac{\cos{(\delta_{k-1,m})}}{l^{m+\frac{7}{2}}}\\
\ge&\ \frac{k^{\frac{13}{2}}\Gamma(k)}{\Gamma(k+\frac{13}{2})}\sum_{m=0}^{3}t_m(3)\frac{(k+\frac{7}{2}+m)_{3-m}}{k^{3-m}l^{m+\frac{7}{2}}}\left(\cos\left(l-(\frac{7}{2}-m)\frac{\pi}{2}\right)-(3+m)\frac{l}{k}\right)\\
\ge &\min\left\{(1-\frac{16}{k})\sum_{m=0}^{3}t_m(3)\frac{(k+\frac{7}{2}+m)_{3-m}}{k^{3-m}l^{m+\frac{7}{2}}}\left(\cos\left(l-(\frac{7}{2}-m)\frac{\pi}{2}\right)-(3+m)\frac{l}{k}\right),0\right\}.
\end{align*}
Write $$(1-\frac{16}{k})\sum_{m=0}^{3}t_m(3)\frac{(k+\frac{7}{2}+m)_{3-m}}{k^{3-m}l^{m+\frac{7}{2}}}\left(\cos\left(l-(\frac{7}{2}-m)\frac{\pi}{2}\right)-(3+m)\frac{l}{k}\right)=\sum_{i=0}^4E_i,$$
where
\begin{align*}
    E_0=\frac{1024 l^3 \cos \left(l+\frac{\pi }{4}\right)-1920 l^2 \cos \left(l-\frac{\pi }{4}\right)-840 l \cos \left(l+\frac{\pi }{4}\right)-315 \cos \left(l-\frac{\pi }{4}\right)}{1024 l^{13/2}},
\end{align*}

\begin{align*}
    E_1=\frac{-3 \left(512 l^3+1280 l^2-2304 l^2 \cos \left(l+\frac{\pi }{4}\right)+700 l-1920 l \cos \left(l-\frac{\pi }{4}\right)+770 \cos \left(l+\frac{\pi }{4}\right)-315\right)}{512 k l^{11/2}},
\end{align*}

\begin{align*}
    E_2=\frac{-10368 l^2+11520 l+15296 l \cos \left(l+\frac{\pi }{4}\right)+64920 \cos \left(l-\frac{\pi }{4}\right)-5775}{256 k^2 l^{9/2}},
\end{align*}

\begin{align*}
    E_3=\frac{-3 \left(478 l^2-2705 l-231 l \cos \left(l+\frac{\pi }{4}\right)-1980 \cos \left(l-\frac{\pi }{4}\right)\right)}{8 k^3 l^{9/2}}
\end{align*}

\begin{align*}
    E_4=\frac{297 (-7 l+80)}{8 k^4 l^{7/2}}.
\end{align*}

If $5\le l\le 6.5$, then $E_0\ge -0.0002$, $E_1\ge -0.00025$, $E_2\ge  -1.5\times 10^{-5}$, $E_3\ge -10^{-7}$, $E_4\ge 0$. By \eqref{tilde R1}, $|\widetilde R|\le 8\times 10^{-6}$. Therefore from \eqref{**} we have
\begin{align*}
    \Tilde{F}_k'(\cos{\zeta})\ge -48\sqrt{\frac{2}{\pi}}\times 0.005\ge -0.04.
\end{align*}

If $l>6.5$, then $E_0\ge -0.00077$, $E_1\ge -0.0002$, $E_2\ge  -10^{-5}$, $E_3\ge -10^{-7}$, $E_4\ge -10^{-9}$. Either \eqref{tilde R1} or \eqref{tiled R2} implies $|\widetilde R|\le 3\times 10^{-7}$, so we also have
\begin{align*}
    \Tilde{F}_k'(\cos{\zeta})\ge -48\sqrt{\frac{2}{\pi}}\times 0.01\ge -0.04.
\end{align*}
Thus the lemma is proved.
\end{proof}

\begin{proof}[Proof of Lemma~\ref{ptwise}]
We first prove the following estimate at one point:
    \begin{equation}\label{pt}
        0.3\le\Tilde{F}_k'(1-\frac{8}{\lambda_k})\le 0.33, \quad k\ge 6.
    \end{equation}
 Direct computation by Matlab shows that \eqref{pt} holds for $6\le k\le 100$, so in what follows we may assume $k>100$. The main tool we use is the hypergeometric expansion \eqref{hyper odd} and \eqref{hyper even}. We will prove \eqref{pt} only for even $k$, and the case for odd $k$ is similar.

 Let $k=2m+2$, then $\Tilde{F}_k'=F_{k-1}^\frac{7}{2}$, so by \eqref{hyper odd},
 \begin{align*}
     \Tilde{F}_k'(1-\frac{8}{\lambda_k})=(1-\frac{8}{\lambda_k}) {_2}F_1(-m,m+\frac{9}{2};4;t),
 \end{align*}
where $t=1-(1-\frac{8}{\lambda_k})^2=\frac{8}{\lambda_k}(2-\frac{8}{\lambda_k})$. Now we write
\begin{align*}
    {_2}F_1(-m,m+\frac{9}{2};4;t)=\sum_{i=0}^m(-1)^i\gamma_it^i,
\end{align*}
where $\gamma_i=\frac{(m-i+1)_i(m+\frac{7}{2})_i}{i!(4)_i}$. It is easy to see that
\begin{align*}
    \min\limits_{1\le i< m}\{ \frac{\gamma_{i}}{\gamma_{i+1}}\}= \frac{\gamma_{1}}{\gamma_{2}}=\frac{10}{(m-1)(m+\frac{9}{2})}=\frac{40}{(k-3)(k+7)}>t.
\end{align*}
Therefore
\begin{align*}
    \sum_{i=0}^{j_1 \text{ is odd}}(-1)^i\gamma_it^i\le {_2}F_1(-m,m+\frac{9}{2};4;t)\le \sum_{i=0}^{j_2 \text{ is even}}(-1)^i\gamma_it^i
\end{align*}
Take $j_1=5, j_2=6$, then direct computation shows that \eqref{pt} holds since $m\ge 50$.

Now in view of Lemma~\ref{minFk}, we see that $\Tilde{F}_k'(1-\frac{8}{\lambda_k})>-\min\limits_{0\le x\le 1} \Tilde{F}_k'(x)$. Then by Lemma~\ref{localmax} $(b)$,  $\Tilde{F}_k'(1-\frac{8}{\lambda_k})\ge \Tilde{F}_k'(x)$ for all $0\le x\le 1-\frac{8}{\lambda_k}$. Moreover, the convexity of $\widetilde{F}_k'(x)$ on $[1-\frac{8}{\lambda_k}, 1]$ is guaranteed by Lemma~\ref{localmax} $(c)$. This completes the proof of Lemma \ref{ptwise}.
\end{proof}

\section{proof of Lemma \ref{nu>1}}\label{appendix cancel}
We first prove a simple lemma, which enables us to focus on the region near $x=1$. By letting $x=\cos\theta$, we introduce the function $v(\theta)=(\sin\theta)^2 F_n^\nu(\cos\theta)$ in this appendix.
\begin{lemma}\label{increase local max}
For $n\ge 2$ and $\nu>0$, let $v(\theta)$ be defined as above. If $\nu\ge 2$, then the successive relative maxima of $|v(\theta)|$ form an increasing sequence as $\theta$ decreases from $\frac{\pi}{2}$ to $0$.
\end{lemma}
\begin{proof}[Proof of Lemma \ref{increase local max}] By \eqref{DE} it is straightforward to check that $v$ satisfies the equation
\begin{equation*}
    v''(\theta)+p(\theta)v'(\theta)+q(\theta)=0,
\end{equation*}
where $p(\theta)=(2\nu-4)\cos\theta$, and $q(\theta)=(n^2+2\nu n+4)-\frac{2}{\sin^2\theta}+(4\nu-8)(\sin\theta-\frac{1}{\sin\theta})$. Since $\nu>2$, we know that $p\ge 0$, $q$ is increasing and $q$ has a unique zero $\widetilde\theta$  in $(0,\frac{\pi}{2})$ .

Since $v(0)=0, v'>0$ near $0$, and $q(\theta)<0$ in $(0,\widetilde\theta)$, by the maximum principle, it's easy to see that $|v(\theta)|$ has no local maxima in $(0,\widetilde\theta]$. Now we consider the case when $\theta\in (\widetilde\theta, \frac{\pi}{2}]$. Let $\Tilde q=q^{-1}$, then $\Tilde q>0$ is strictly decreasing in $(\widetilde\theta, \frac{\pi}{2}]$. Introducing
\begin{equation*}
    f(\theta)=v^2(\theta)+\Tilde q(\theta)(v')^2(\theta),
\end{equation*}
we have
\begin{equation*}
    f'=\Tilde q'(v')^2+2v'(\Tilde{q}v''+v)=(\Tilde{q}'-2p\Tilde{q})u'^2<0.
\end{equation*}
    But  $f(\theta)=v^2(\theta)$ if $v'(\theta)=0$, so the lemma is proved.
    \end{proof}

\begin{proof}[Proof of Lemma \ref{nu>1}] In view of Lemma \ref{increase local max}, we need to find a bound for $\theta_*$, the smallest zero of $v'(\theta)$ in $(0, \frac{\pi}{2})$. By definition of $v$ and \eqref{derivative},
\begin{align*}
    v'(\theta)&=\sin\theta\left(2\cos\theta F_n^\nu(\cos\theta)-\sin^2\theta (F_n^\nu)'(\cos\theta)\right)\\
    &=\sin\theta\left(2 \cos\theta F_n^\nu(\cos\theta)-  \frac{n(n+2\nu)}{2\nu+1}\sin^2\theta F_{n-1}^{\nu+1}(\cos\theta)\right).
\end{align*}
We claim that when $\theta=\overline\theta=\arcsin{\sqrt{\frac{4\nu+2}{n(n+2\nu)}}}$,
\begin{align}\label{v'-}
     v'(\overline{\theta})=2\sin\overline{\theta}\left(\cos \overline\theta F_{n}^\nu(\cos\overline\theta)-F_{n-1}^{\nu+1}(\cos\overline\theta)\right)<0.
\end{align}
We will use the hypergeometric function expansion for Gegenbauer polynomials \eqref{hyper odd} and \eqref{hyper even} to prove \eqref{v'-}. We only give the proof
 for odd $n$, and the proof
 for even $n$ is similar.

Write $n=2m+1$. By Lemma \ref{xn1}, it is not difficult to show $\cos \overline\theta>x_{2m+1,1}(\nu)$, hence $F_{2m+1}^{\nu}(\cos \overline\theta)>0$, so we have
    \begin{align*}
      F_{2m}^{\nu+1}(\cos\overline\theta)-\cos\overline\theta F_{2m+1}^\nu(\cos\overline\theta)
      &\ge {_2}F_1(-m,m+\nu+1; \nu+\frac{3}{2}; \sin^2\overline\theta)- {_2}F_1(-m,m+\nu+1; \nu+\frac{1}{2}; \sin^2\overline\theta)\\
      &=\sum_{k=1}^m (-1)^{k+1} \alpha_k(\sin^2\overline\theta)^{k},
    \end{align*}
    where $\alpha_k=\frac{(m-k+1)_k(m+\nu+1)_k}{(k-1)!(\nu+\frac{1}{2})_{k+1}}$. We compute
    \begin{align*}
        \frac{\alpha_{k}}{\alpha_{k+1}}=\frac{k(k+\nu+\frac{3}{2})}{(m-k)(m+\nu+k+1)}.
    \end{align*}
    It is then easy to see that $$\min\limits_{1\le k< m}\{ \frac{\alpha_{k}}{\alpha_{k+1}}\}= \frac{\alpha_{1}}{\alpha_{2}}=\frac{\nu+\frac{5}{2}}{(m-1)(m+\nu+2)}.$$
    Since $\sin^2\overline\theta=\frac{4\nu+2}{n(n+2\nu)}=\frac{4\nu+2}{(2m+1)(2m+2\nu+1)}<\frac{\nu+\frac{5}{2}}{(m-1)(m+\nu+2)}$, no matter $m$ is even or odd,  we have
\begin{align*}
     F_{2m}^{\nu+1}(\cos\overline\theta)-\cos \overline\theta F_{2m+1}^\nu(\cos\overline\theta)\ge \sum_{\substack{1\le k\le m\\k \text{ is odd}}} (\sin^2\overline\theta)^k (\alpha_k-\alpha_{k+1}\sin^2\overline\theta)>0,
\end{align*}
    where $\alpha_{m+1}=0 $ is understood, so \eqref{v'-} holds. Consequently, since $v'(\theta)>0$ when $\theta$ is small, from \eqref{v'-} we know that $\theta_*<\overline\theta$.

Now we look for a lower bound of $\theta_*$. Let $\underline\theta=\arcsin\sqrt{\frac{4\nu+2}{n(n+2\nu)}\delta}$, where $0<\delta<1$ is to be determined. We want to show that
\begin{align}\label{v'+}
v'(\theta)=2\sin\theta\left( \cos\theta F_n^\nu(\cos\theta)-\delta F_{n-1}^{\nu+1}(\cos\theta)\right)>0
\end{align}
for all $0\le \theta< \underline\theta.$
As before, we only consider the case $n=2m+1$, then we can write
\begin{align*}
    \cos\theta F_n^\nu(\cos\theta)-\delta F_{n-1}^{\nu+1}(\cos\theta)=\sum_{k=0}^m(-1)^k\beta_k(\sin^2\theta)^k,
\end{align*}
where
\begin{equation*}
    \beta_k=\frac{(m-k+1)_{k}(m+\nu+1)_{k}}{k!(\nu+\frac{1}{2})_{k+1}}\left((\nu+\frac{1}{2}+k)\cos^2\theta)-\delta(\nu+\frac{1}{2}) \right).
\end{equation*}
We compute
\begin{align*}
    \frac{\beta_k}{\beta_{k+1}}=\frac{(k+1)(\nu+\frac{1}{2}+k+1)}{(m-k)(m+\nu+k+1)}\frac{(\nu+\frac{1}{2}+k)\cos^2\theta-\delta(\nu+\frac{1}{2}) }{(\nu+\frac{3}{2}+k)\cos^2\theta-\delta(\nu+\frac{1}{2}) },
\end{align*}
so
\begin{align*}
    \min\limits_{0\le k< m}\{ \frac{\beta_{k}}{\beta_{k+1}}\}= \frac{\beta_{0}}{\beta_{1}}=\frac{\nu+\frac{3}{2}}{m(m+\nu+1)}\frac{(\nu+\frac{1}{2})\cos^2\theta-\delta(\nu+\frac{1}{2}) }{(\nu+\frac{3}{2})\cos^2\theta-\delta(\nu+\frac{1}{2}) }.
\end{align*}
Therefore to prove \eqref{v'+}, it is enough to show $\frac{\beta_0}{\beta_1}>\sin^2\theta$, or equivalently
\begin{align*}
    (\nu+\frac{3}{2})(\nu+\frac{1}{2})(\cos^2\theta-\delta)>m(m+\nu+1)\left((\nu+\frac{3}{2})\cos^2\theta-\delta(\nu+\frac{1}{2})\right)\sin^2\theta.
\end{align*}
This is a quadratic inequality about $\sin^2\theta$. If we choose
\begin{align}\label{delta}
    \delta=\frac{\nu-\sqrt{\nu}+\frac{1}{2}}{\nu+\frac{1}{2}},
\end{align}
    then since we have assumed that $n\ge 2\nu+2$, direct computation shows that it is enough to prove the above inequality for $\theta=\underline\theta$, which reduces to
\begin{align*}
    (\nu+\frac{3}{2})(\cos^2\underline\theta-\delta)>m(m+\nu+1)\left((\nu+\frac{3}{2})\cos^2\underline\theta-\delta(\nu+\frac{1}{2})\right)\frac{4\delta}{n(n+2\nu)}.
\end{align*}
Since $\frac{4m(m+\nu+1)}{n(n+2\nu)}=\frac{(n-1)(n+2\nu+1)}{n(n+2\nu)}<1$, we only need to show
\begin{align*}
    \cos^2\underline\theta-\delta>\delta\left(\cos^2\underline\theta-\frac{\nu+\frac{1}{2}}{\nu+\frac{3}{2}}\delta\right),
\end{align*}
which is easy to verify, so we omit the details.

From \eqref{v'-} and \eqref{v'+}, we have $\underline\theta<\theta_*<\overline{\theta}$, so
\begin{align}\label{theta*}
    |v(\theta_*)|=|\sin^2\theta_* F_n^\nu(\theta_*)|\le |\sin^2\overline\theta F_n^\nu(\underline\theta)|=\frac{4\nu+2}{n(n+2\nu)}F_n^\nu(\underline\theta).
\end{align}
It remains to give an upper bound for $F_n^\nu(\underline\theta)$. Let $n=2m+1$, then
\begin{align*}
    F_n^\nu(\underline\theta)&=\cos\underline\theta\sum_{k=0}^m\frac{(-1)^k(m-k+1)_{k}(m+\nu+1)_{k}}{k!(\nu+\frac{1}{2})_{k}}\sin^{2k}\underline\theta\\
    &\le \sum_{k=0}^{l \text{ is even}}\frac{(-1)^k(m-k+1)_{k}(m+\nu+1)_{k}}{k!(\nu+\frac{1}{2})_{k}}\sin^{2k}\underline\theta.
\end{align*}
For $m\ge 5$, we can choose $l=4$ to obtain
\begin{align}\label{up theta-}
    F_n^\nu(\underline\theta)&\le \sum_{k=0}^{4}\frac{(-1)^k(m-k+1)_{k}(m+\nu+1)_{k}}{k!(\nu+\frac{1}{2})_{k}}\left(\frac{4\nu+2}{n(n+2\nu)}\right)^k\delta^k\notag\\
    &=\sum_{k=0}^{4}\frac{(-1)^k(m-k+1)_{k}(m+\nu+1)_{k}}{k!(\nu+\frac{1}{2})_{k}}\left(\frac{\nu-\sqrt{\nu}+\frac{1}{2}}{(m+\frac{1}{2})(m+\nu+\frac{1}{2})}\right)^k
\end{align}
Direct computation shows that for fixed $\nu$, then above expression, viewed as a function of $m>\nu$, is decreasing in $m$. Therefore if $n\ge \max\{2\nu+2, 12\}$, \eqref{theta*} and \eqref{up theta-} together imply that
\begin{align*}
    |v(\theta_*)|\le \frac{\widetilde C_\nu}{n(n+2\nu)},
\end{align*}
where
\begin{equation}\label{Cnu}
    \widetilde C_\nu=\left\{\begin{aligned} &(4\nu+2)\sum_{k=0}^{4}\frac{(-1)^k(6-k)_{k}(6+\nu)_{k}}{k!(\nu+\frac{1}{2})_{k}}\left(\frac{\nu-\sqrt{\nu}+\frac{1}{2}}{\frac{11}{2}(\nu+\frac{11}{2})}\right)^k, &\text{if $\nu<5$},\\
    &(4\nu+2)\sum_{k=0}^{4}\frac{(-1)^k(\nu-k+1)_{k}(2\nu+1)_{k}}{k!(\nu+\frac{1}{2})_{k}}\left(\frac{\nu-\sqrt{\nu}+\frac{1}{2}}{(\nu+\frac{1}{2})(2\nu+\frac{1}{2})}\right)^k, &\text{if $\nu\ge5$}.
    \end{aligned}\right.
\end{equation}

We remark that same estimates holds for even $n$. Finally, since
\begin{equation*}
    |v(\frac{\pi}{2})|=F_n^\nu(0)=\left\{ \begin{aligned}
        &0, &\text {if $n$ is odd},\\
        &\frac{\Gamma(\frac{n}{2}+\nu)}{\Gamma(\nu)(\frac{n}{2})!}\Big/\frac{\Gamma(n+2\nu)}{\Gamma(2\nu)n!}=\frac{2\Gamma(\nu+\frac{1}{2})\Gamma(\frac{n+3}{2})}{(n+1)\sqrt{\pi}\Gamma(\frac{n+1}{2}+\nu)},&\text{ if $n$ is even},
    \end{aligned}\right.
\end{equation*}
we conclude that
\begin{equation*}
    |v(\theta)|\le \max\{|v(\theta_*)|, |v(\frac{\pi}{2})|\}\le \frac{\widetilde C_\nu}{n(n+2\nu)}.
\end{equation*}
\end{proof}

\section{proof of Proposition \ref{lambda=1}}\label{appendix lambda==1}
\begin{proof}[Proof of Proposition \ref{lambda=1}] If $k=2$ or $4$, then by Lemma \ref{A2-5}, one can check the proposition holds true for all $n\ge 6$ directly, so in what follows we may assume $k\ge 6$.

We first consider the case when $n\ge 65$. Recall that $d=8$, $b=0.33$ are given in Theorem \ref{bk},  $d_0=17$ ,and $B_k=\frac{9\alpha^2}{32}(\lambda_{n+1}-\lambda_k+\frac{11}{7\alpha})(2k+5)$, so we have
\begin{equation}\label{quotient}
    \frac{B_{k+1}-B_k}{B_{k+1}+B_k}=\frac{ \left(n^2+7n-3k^2-18 k-15\right)+\frac{11}{7\alpha}}{(k+3) \left( \left(2 n^2+14 n-2 k^2-12 k+5\right)+\frac{22}{7\alpha}\right)}.
\end{equation}

\noindent Case $1$:  $\lambda_7\le \lambda_{k+1}\le
 \frac{d\lambda_{n}}{2d_0}$.
 In this case, $6\le k\le \frac{n}{2}-1$,  hence $B_{k+1}>B_k$, and by \eqref{quotient}, one can show that $\frac{B_{k+1}-B_k}{B_{k+1}+B_k}$ is decreasing in $k$, so we have
 \begin{align}\label{quotient bd1}
     \frac{B_{k+1}-B_k}{B_{k+1}+B_k}\le\frac{\left(n^2+7 n-231\right)+\frac{11}{7\alpha}}{9 \left( \left(2 n^2+14 n-139\right)+\frac{22}{7 \alpha}\right)}<0.054.
 \end{align}
 Moreover, $a\le \frac{d_0}{\lambda_n}\le \frac{d}{2\lambda_ {k}}$. so \eqref{R1} becomes
\begin{align*}
    R_{k,1}&\le B_k\Big((a_+-\frac{\lambda_k}{d}(1-b)a_+^2)^2+(a_--\frac{\lambda_k}{d}(1-b)a_-^2)^2\Big)^2\\
    &+B_{k+1}\Big((a_+-\frac{\lambda_{k+1}}{d}(1-b)a_+^2)^2+(a_--\frac{\lambda_{k+1}}{d}(1-b)a_-^2)^2\Big)^2\\
    &=B_k\Big((2\lambda^2-2\lambda+1)a^2-\frac{2\lambda_k}{d}(1-b)(1-3\lambda+3\lambda^2)a^3+(\frac{\lambda_k}{d}(1-b))^2(\lambda^4+(1-\lambda)^4)a^4\Big)\\
    &+B_{k+1}\Big((2\lambda^2-2\lambda+1)a^2-\frac{2\lambda_{k+1}}{d}(1-b)(1-3\lambda+3\lambda^2)a^3+(\frac{\lambda_{k+1}}{d}(1-b))^2(\lambda^4+(1-\lambda)^4)a^4\Big),
\end{align*}
and \eqref{R3} becomes
\begin{align*}
    R_{k,3}\le 2(B_{k+1}-B_{k})\lambda(1-\lambda)a^2.
\end{align*}
Combined with \eqref{R2}, we can write
\begin{align*}
    f_{k,a}(\lambda)&=B_k\Big((2\lambda^2-2\lambda+1)-\frac{2\lambda_k}{d}(1-b)(1-3\lambda+3\lambda^2)a+(\frac{\lambda_k}{d}(1-b))^2(\lambda^4+(1-\lambda)^4)a^2\Big)\\
    &+B_{k+1}\Big((2\lambda^2-2\lambda+1)-\frac{2\lambda_{k+1}}{d}(1-b)(1-3\lambda+3\lambda^2)a+(\frac{\lambda_{k+1}}{d}(1-b))^2(\lambda^4+(1-\lambda)^4)a^2\Big)\\
    &+(2c_k(B_k+B_{k+1})+2(B_{k+1}-B_{k}))\lambda(1-\lambda).
\end{align*}
For $\frac{1}{2}\le \lambda<1$, direct computation yields
\begin{align*}
    \frac{f_{k,a}(1)-f_{k,a}(\lambda)}{2(\lambda-\lambda^2)}&=B_k\Big(1-\frac{3\lambda_k}{d}(1-b)a+(\frac{\lambda_k}{d}(1-b))^2a^2(\lambda^2-\lambda+2)\Big)\\
    &+B_{k+1}\Big(1-\frac{3\lambda_{k+1}}{d}(1-b)a+(\frac{\lambda_{k+1}}{d}(1-b))^2a^2(\lambda^2-\lambda+2)\Big)
    -c_k(B_{k+1}+B_k)-(B_{k+1}-B_k)\\
    &\ge B_k(1-3w_k+\frac{7}{4}w_k^2)+B_{k+1}(1-3w_{k+1}+\frac{7}{4}w_{k+1}^2)-c_k(B_{k+1}+B_k)-(B_{k+1}-B_k),
\end{align*}
where
\begin{equation*}
    w_j=\frac{\lambda_j}{d}(1-b)a\le \frac{1-b}{2}<\frac{6}{7},\ j=k,k+1.
\end{equation*}
So by Corollary \ref{cn} and \eqref{quotient bd1}
\begin{align*}
    \frac{f_{k,a}(1)-f_{k,a}(\lambda)}{2(\lambda-\lambda^2)}&\ge (\frac{7 b^2+10 b-1}{16} -c_k)(B_k+B_{k+1})-(B_{k+1}-B_k)\\
    &\ge (B_k+B_{k+1})(0.191-0.12-0.054)
    >0.
\end{align*}

\noindent Case $2$: $\frac{\lambda_n}{4}=\frac{d\lambda_n}{2d_0}<\lambda_{k+1}\le \lambda_n$, but $a_-\le\frac{d}{2\lambda_{k+1}}$.
In this case, $\frac{n}{2}-2\le k\le n-1$, $\lambda\ge1-\frac{d}{2\lambda_ka}$, and  we have
\begin{align*}
    R_{k,1}&\le B_k\left((ba_++\frac{d}{4\lambda_k}(1-b))^2+(a_--\frac{\lambda_k}{d}(1-b)a_-^2)^2 \right)\\
    &+B_{k+1}\left((ba_++\frac{d}{4\lambda_{k+1}}(1-b))^2+(a_--\frac{\lambda_{k+1}}{d}(1-b)a_-^2)^2 \right).
\end{align*}
Since the sign of $B_{k+1}-B_k$ is unknown, we need to discuss both cases separately.

If $ B_{k+1}\le B_k$, then by \eqref{quotient},
\begin{align}\label{quotient max}
    \frac{B_k-B_{k+1}}{B_{k}+B_{k+1}}\le \frac{7 \alpha(2 n+5)-11}{(n+2) (21\alpha(2 n+5)+22)}<\frac{1}{3},
\end{align}
and we have
\begin{align*}
    R_{k,3}\le 2(B_{k}-B_{k+1})m_0(1-\lambda)a^2.
\end{align*}

Combined with \eqref{R2}, for $\frac{1}{2}\le \lambda< 1$, we have
\begin{align*}
    \phi(\lambda):=\frac{f_{k,a}(1)-f_{k,a}(\lambda)}{(1-\lambda) a^2}
    &=B_k\left((1+\lambda)b^2+\frac{d}{2\lambda_ka}(1-b)b-(1-\lambda)(1-\frac{\lambda_k}{d}(1-b)(1-\lambda) a)^2 \right)\\
    &+B_{k+1}\left((1+\lambda)b^2+\frac{d}{2\lambda_{k+1}a}(1-b)b-(1-\lambda)(1-\frac{\lambda_{k+1}}{d}(1-b)(1-\lambda) a)^2 \right)\\
    &-2c_k(B_k+B_{k+1})\lambda
    -2(B_{k}-B_{k+1})m_0.
\end{align*}

Then
\begin{align*}
    \phi'(\lambda)&=B_k\left(b^2+\left(1-\frac{\lambda_k}{d}(1-b)(1-\lambda) a\right)\left(1-\frac{3\lambda_k}{d}(1-b)(1-\lambda) a\right)\right)\\
    &+B_{k+1}\left(b^2+\left(1-\frac{\lambda_{k+1}}{d}(1-b)(1-\lambda) a\right)\left(1-\frac{3\lambda_{k+1}}{d}(1-b)(1-\lambda) a\right)\right)-2c_k(B_k+B_{k+1}).
\end{align*}
By assumption $\frac{\lambda_k}{d}(1-b)(1-\lambda) a\le \frac{\lambda_{k+1}}{d}(1-b)(1-\lambda) a=\frac{\lambda_{k+1}}{d}(1-b)a_-\le \frac{1-b}{2}$, so by Corollary \ref{cn},
\begin{align*}
    \phi'(\lambda)\ge (B_k+B_{k+1})\left(b^2-2c_k+\frac{(b+1)(3b-1)}{4}\right)>(B_k+B_{k+1})\left(0.105-2c_k\right)>0.
\end{align*}
Since $\lambda\ge 1-\frac{d}{2\lambda_{k+1} a}$, we need to discuss the following two cases:

\noindent If $\frac{d}{2\lambda_{k+1} a}\ge \frac{1}{2}$, then the lower bound of $\lambda$ is $\frac{1}{2}$. Moreover, $\lambda_{k+1}\le \frac{d}{a}\le \frac{\lambda_{n+4}d}{d_0}=\frac{\lambda_{n+4}}{2}$, so $k\le \frac{n}{\sqrt{2}}+2$. Consequently, from \eqref{quotient} it's easy to check that
\begin{align*}
    \frac{B_k-B_{k+1}}{B_{k}+B_{k+1}}<0.008,
\end{align*}
Therefore by Lemma \ref{minFk} and Corollary \ref{cn}, we have
\begin{align*}
    \phi(\lambda) \ge \phi(\frac{1}{2})
    &=B_k\left(\frac{3}{2}b^2+\frac{d}{2\lambda_ka}(1-b)b-\frac{1}{2} \left(1-\frac{\lambda_k}{2d}(1-b)a\right)^2\right)\\
    &+B_{k+1}\left(\frac{3}{2}b^2+\frac{d}{2\lambda_{k+1}a}(1-b)b-\frac{1}{2} \left(1-\frac{\lambda_{k+1}}{2d}(1-b)a\right)^2\right)\\
    &-(B_k+B_{k+1})c_k-2m_0({B_{k}-B_{k+1}})\\
    &\ge (B_{k}+B_{k+1})(0.02746-c_k-0.016m_0)\\
    &>0.
\end{align*}

\noindent If $\frac{d}{2\lambda_{k+1} a}\le \frac{1}{2}$, then the lower bound of $\lambda$ is $1-\frac{d}{2\lambda_{k+1} a}$, so by \eqref{quotient max}, Lemma \ref{minFk} and Corollary \ref{cn}, we have
\begin{align*}
    \phi(\lambda)\ge\phi(1-\frac{d}{2\lambda_{k+1}}a)&=B_k\left((2-\frac{d}{2\lambda_{k+1}a})b^2+\frac{d}{2\lambda_k a}(1-b)b-\frac{d}{2\lambda_{k+1}a}(1-\frac{\lambda_k}{\lambda_{k+1}}\frac{1-b}{2})^2\right)\\
    &+B_{k+1}\left((2-\frac{d}{2\lambda_{k+1}a})b^2+\frac{d}{2\lambda_{k+1} a}(1-b)b-\frac{d}{2\lambda_{k+1}a}(1-\frac{1-b}{2})^2\right)\\
    &-(B_k+B_{k+1})c_k-2m_0(B_k-B_{k+1})\\
    &\ge (B_k+B_{k+1})\left(\frac{3}{2}b^2+\frac{(1-b)b}{2}-\frac{1}{2}(1-\frac{\lambda_k}{\lambda_{k+1}}\frac{1-b}{2})^2-c_k-2m_0\frac{B_k-B_{k+1}}{B_k+B_{k+1}}\right)\\
    &\ge (B_k+B_{k+1})(0.05-c_k-\frac{2}{3}m_0)\\
    &>0.
\end{align*}

If {$B_k<B_{k+1}$}, then
$\frac{n}{2}-2\le k \le\frac{n}{\sqrt{3}}$, so
\begin{align}\label{quotient bd2}
   \frac{B_{k+1}-B_{k}}{B_{k}+B_{k+1}}\le \frac{7 \alpha \left(n^2+16 n+36\right)+44}{(n+2) \left(21\alpha\left(n^2+8 n+14\right)+44\right)}\le 0.004.
\end{align}
In this case, we have
\begin{align*}
    R_{k,3}\le 2(B_{k+1}-B_{k})\lambda(1-\lambda)a^2.
\end{align*}
Then one can go through the same argument as before to prove that $f_{k,a}(1)\ge f_{k,a}(\lambda)$ for $\frac{1}{2}\le\lambda\le 1$. The details are omitted.\\

\noindent Case $3$: $\frac{\lambda_n}{4}=\frac{d\lambda_n}{2d_0}<\lambda_{k+1}\le \lambda_{n}$, and $a_->\frac{d}{2\lambda_{k+1}}$.
In this case $4(1-\lambda)\lambda_{k+1}>\lambda_n$, so $\frac{1}{2}\le \lambda<\frac{3}{4}$,  and $2\lambda_{k+1}>\lambda_n$. Hence $k\ge \frac{n-2}{\sqrt{2}}$ and $B_k\ge B_{k+1}$.
Now \eqref{R1} and \eqref{R3} becomes
\begin{align*}
    R_{k,1}&\le B_k\left((ba_++\frac{d(1-b)}{4\lambda_k})^2+(ba_-+\frac{d(1-b)}{4\lambda_k})^2 \right)+B_{k+1}\left((ba_++\frac{d(1-b)}{4\lambda_k})^2+(ba_-+\frac{d(1-b)}{4\lambda_k})^2 \right)\\
    &=B_k\left((2\lambda^2-2\lambda+1)b^2a^2+\frac{4ab(1-b)}{\lambda_k}+8(\frac{1-b}{\lambda_k})^2\right)\\
    &+B_{k+1}\left((2\lambda^2-2\lambda+1)b^2a^2+\frac{4ab(1-b)}{\lambda_{k+1}}+8(\frac{1-b}{\lambda_{k+1}})^2\right)
\end{align*}
and
\begin{align*}
    R_{k,3}\le 2(B_{k}-B_{k+1})m_0(1-\lambda)a^2
\end{align*}
respectively. With the help of \eqref{R2}, after some computations, we deduce that
\begin{align*}
\frac{f_{k,a}(1)-f_{k,a}(\lambda)}{a^2}
    &=B_{k}\left((2\lambda-2\lambda^2)b^2-4(\frac{1-b}{\lambda_ka})^2\right)+B_{k+1}\left((2\lambda-2\lambda^2)b^2-4(\frac{1-b}{\lambda_{k+1}a})^2\right)\\
    &-2(B_k+B_{k+1})c_k\lambda(1-\lambda)-2(B_{k}-B_{k+1})m_0(1-\lambda).
\end{align*}
It's easy to see that for fixed $k$, the above function is increasing in $\lambda$, so
\begin{align*}
    \frac{f_{k,a}(1)-f_{k,a}(\lambda)}{a^2}&\ge B_{k}\left(\frac{1}{2}b^2-4(\frac{1-b}{\lambda_ka})^2\right)+B_{k+1}\left(\frac{1}{2}b^2-4(\frac{1-b}{\lambda_{k+1}a})^2\right)
    -\frac{1}{2}(B_k+B_{k+1})c_k-(B_{k}-B_{k+1})m_0\\
    &\ge (B_k+B_{k+1})\left(\frac{b^2-c_k}{2}-4(\frac{1-b}{\lambda_ka})^2-m_0 \frac{B_k-B_{k+1}}{B_k+B_{k+1}}\right)\\
    &\ge (B_k+B_{k+1})\left( 0.04-\frac{1.7956}{\lambda_k^2a^2}-0.04 \frac{B_k-B_{k+1}}{B_k+B_{k+1}}\right)\\
    &\ge  (B_k+B_{k+1})\left( 0.04-\left(0.0071(\frac{\lambda_{n+4}}{\lambda_k})^2+0.04 \frac{B_k-B_{k+1}}{B_k+B_{k+1}}\right)\right).
\end{align*}

Direct computation shows that $0.0071\frac{\lambda_{n+4}}{\lambda_k}+0.04 \frac{B_k-B_{k+1}}{B_k+B_{k+1}}$ is decreasing in $k$ when $\frac{n-1}{\sqrt{2}}\le k\le n$, therefore
\begin{align*}
    \frac{f_{k,a}(1)-f_{k,a}(\lambda)}{a^2}\ge (B_k+B_{k+1})(0.04-0.035)>0.
\end{align*}

To sum up, by now we have proved Proposition \ref{lambda=1} when $n\ge 65$. When $n<65$, above arguments fail since $c_k$ (hence $R_{k,2}$) is no longer small enough. In this case, we keep  $R_{k,2}$ aside and consider only $R_{k,1}$ and $R_{k,3}$. Then the same argument as above shows that $R_{k,3}$ can be absorbed, which completes the proof. The details are omitted.

\end{proof}

\section{Proof for small $n$}\label{proofsmall}
In the proof of Corollary \ref{cn} and Theorem \ref{main}, we argue for $n$ sufficiently large. In this appendix, we give the numerical data to prove the corresponding cases when $n$ is small.

We first prove Corollary \ref{cn} for small $n$
\begin{proof}[Proof of Corollary \ref{cn} for $30\leq n\leq 428$]
We can use Matlab to calculate the values of $c_n$'s, which are listed as scatter diagrams as follows.
\begin{figure}[htbp]
\centering
\begin{minipage}[t]{0.48\textwidth}
\centering
\includegraphics[width=6cm]{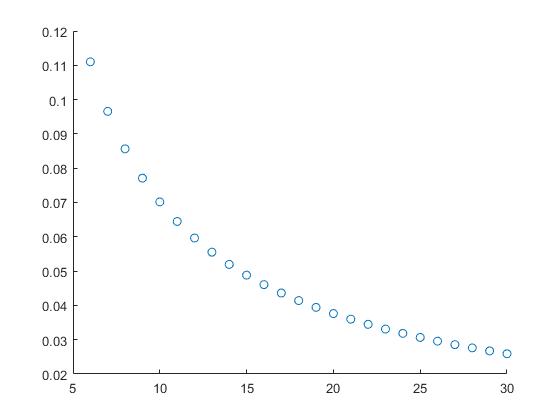}
\caption{$c_n$ for $6\leq n\leq 30$}
\label{figure1}
\end{minipage}
\begin{minipage}[t]{0.48\textwidth}
\centering
\includegraphics[width=6cm]{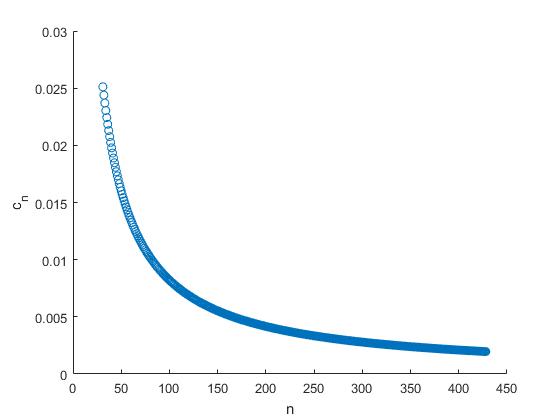}
\caption{$c_n$ for $30\leq n\leq 428$}
\label{figure2}
\end{minipage}
\end{figure}

\end{proof}

Then we give the proof of Theorem \ref{main} when $n$ is small.
\begin{proof}[Proof of Theorem \ref{main} for $n<10000$]
We follow the argument in Section \ref{proofmain}. We only prove for $n\geq 65$ (For the case when $5\leq n\leq 61$, we can use similar methods to run the induction procedure).

Applying Proposition \ref{lambda=1} and plugging it into \eqref{408}, we have

\begin{align}
0\leq&\frac{256}{35}(7-\frac{1}{\alpha})(\frac{27}{7\alpha}-15-\lambda_{n+1})\frac{1}{\alpha}(1-\frac{7}{6}a)+\frac{128}{7}(\lambda_{n+1}-6+\frac{71}{7\alpha})\frac{1}{\alpha^2}(1-\frac{7}{6}a)^2\nonumber\\
+&\frac{176}{63}\alpha(\lambda_2+4)(\lambda_2+6)b_{2}^{2}\nonumber\\
+&\frac{32}{9\alpha^2}\sum_{k=2}^{\frac{n-3}{2}}(\lambda_{n+1}-\lambda_{k}+\frac{11}{7\alpha})(2k+5)(1-\frac{1-b}{d}\lambda_{k}a)^2a^2\nonumber\\
+&\frac{32}{9\alpha^2}\sum_{k=\frac{n-1}{2}}^{n}(\lambda_{n+1}-\lambda_{k}+\frac{11}{7\alpha})(2k+5)(ba+(1-b)\frac{d}{4\lambda_{k}})^2.\nonumber\\
\leq &-\frac{512}{7}(\lambda_{n+1}+\frac{51}{7})(1-\frac{7}{6}a)+\frac{512}{7}(\lambda_{n+1}+\frac{100}{7})(1-\frac{7}{6}a)^2+\frac{22528}{63\alpha}a^2\nonumber\\
+&\frac{128}{9}\sum_{k=2}^{\frac{n-3}{2}}(\lambda_{n+1}-\lambda_{k}+\frac{22}{7})(2k+5)[(1-\frac{1-b}{d}\lambda_{k}a)^2+\frac{1}{2}c_k\chi_{\{5\leq n\leq 61\}}]a^2\nonumber\\
+&\frac{128}{9}\sum_{k=\frac{n-1}{2}}^{n}(\lambda_{n+1}-\lambda_{k}+\frac{22}{7})(2k+5)[(ba+(1-b)\frac{d}{4\lambda_{k}})^2+\frac{1}{2}c_k\chi_{\{5\leq n\leq 61\}}]a^2\nonumber\\
=&:\tilde{g}_n(a).
 \end{align}

To obtain a contradiction, it suffices to show that $\tilde{g}_{n}(a)$ is negative for $\frac{16}{\lambda_{n+4}}<a\leq \frac{16}{\lambda_n}$, for any $n<10000$ with $n\equiv 1\text{ (mod }4)$. Note that $\tilde{g}_n(a)$ is a parabola of $a$ with positive constant term. It suffices to show $\tilde{g}_{n}(\frac{16}{\lambda_{n+4}})$ and $\tilde{g}_{n}(\frac{16}{\lambda_{n}})$ are negative. Using Matlab, we obtain the following scatter diagrams for the above two quantities and thus we are done.

\begin{figure}[htbp]
\centering
\begin{minipage}[t]{0.48\textwidth}
\centering
\includegraphics[width=6cm]{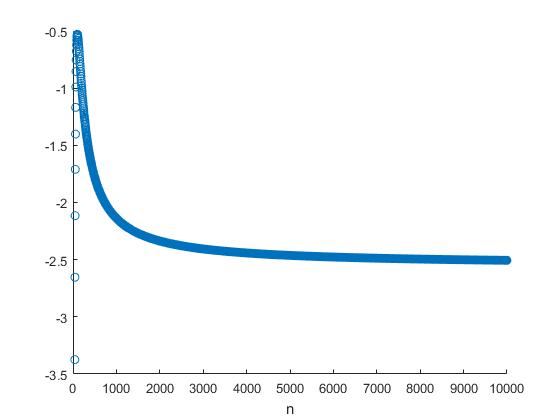}
\caption{$\tilde{g}_{n}(\frac{16}{\lambda_{n+4}})$ from $n=41$ to $9997$}
\label{figure1}
\end{minipage}
\begin{minipage}[t]{0.48\textwidth}
\centering
\includegraphics[width=6cm]{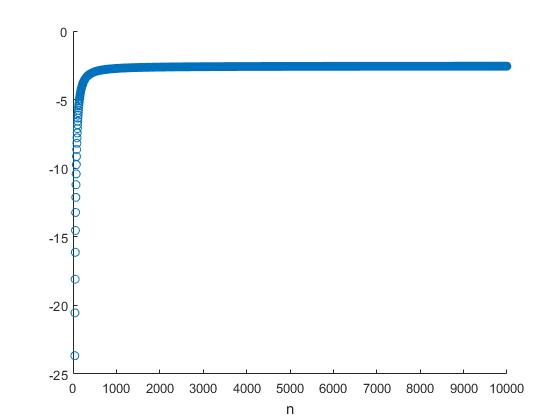}
\caption{$\tilde{g}_{n}(\frac{16}{\lambda_{n}})$ from $n=41$ to $9997$}
\label{figure2}
\end{minipage}
\end{figure}

\end{proof}

\medskip

\section*{Acknowledgements}

The research of J. Wei was partially supported by NSERC of Canada.  The research of  C.Gui was partially supported by NSF award  DMS-2155183 and a UMDF Professorial Fellowship of University of Macau.

\end{document}